\documentclass[10pt]{amsart}
\usepackage{amsmath,amssymb}
\usepackage{amsthm}
\usepackage{indentfirst}
\usepackage{mathrsfs}
\usepackage{graphicx}
\usepackage{amsthm}
\usepackage{thmtools}
\usepackage{mathtools}
\usepackage{graphicx}
\usepackage{subcaption}
\usepackage[usenames, dvipsnames]{color}
\usepackage{physics}
\usepackage{csquotes}
\usepackage{verbatim}
\usepackage{bookmark}
\usepackage{booktabs}
\usepackage{dsfont}
\usepackage{xcolor}
\usepackage{geometry}
\usepackage{enumitem}
\usepackage{afterpage}
\usepackage{nicefrac}
\DeclarePairedDelimiter\ceil{\lceil}{\rceil}
\usepackage{hyperref}
\hypersetup{
colorlinks=true,
citecolor=blue,
urlcolor=blue,
linkcolor=blue,
}
\usepackage{pifont}
\newcommand{\cmark}{\textcolor{green!80!black!70!}{\text{\ding{51}}}}
\newcommand{\xmark}{\textcolor{red!70!black!80!}{\text{\ding{55}}}}

\usepackage[foot]{amsaddr}
\usepackage[nobysame,initials]{amsrefs}

\newtheorem{theorem}{Theorem}
\newtheorem{lemma}[theorem]{Lemma}
\newtheorem{proposition}[theorem]{Proposition}
\theoremstyle{remark}
\newtheorem{remark}[theorem]{Remark}

\newcommand{\thmref}[1]{Theorem~\ref{#1}}
\newcommand{\propref}[1]{Proposition~\ref{#1}}
\newcommand{\lemref}[1]{Lemma~\ref{#1}}
\newcommand{\secref}[1]{\S~\ref{#1}}
\newcommand{\figref}[1]{Figure~\ref{#1}}
\newcommand{\appref}[1]{Appendix~\ref{#1}}
\newcommand{\ie}{{i.e.}}
\newcommand{\eg}{{e.g.}}

\makeatletter
\newcommand*{\rom}[1]{\expandafter\@slowromancap\romannumeral #1@}
\makeatother

\usepackage{multirow, makecell}
\newcommand\myeq[1]{\mathrel{\stackrel{\makebox[0pt]{\mbox{\normalfont\tiny #1}}}{=}}}
\newcommand\myle[1]{\mathrel{\stackrel{\makebox[0pt]{\mbox{\normalfont\tiny #1}}}{\le}}}

\newcommand{\Natural}{\mathbb{N}}
\newcommand{\Real}{\mathbb{R}}
\newcommand{\ee}{\mathbb{E}}

\def\bigl{\mathopen\big}
\def\bigr{\mathclose\big}

\newcommand{\unit}{\mathbb{I}}
\newcommand{\id}{\mathcal{I}}
\newcommand{\hbt}{\mathfrak{H}}
\newcommand{\dt}{\Delta t}
\newcommand{\normbig}[1]{\big\lVert\,#1\,\big\rVert}
\newcommand{\normBig}[1]{\Big\lVert\,#1\,\Big\rVert}

\newcommand{\ttnorm}[1]{\norm{#1}_1}
\newcommand{\ttnormbig}[1]{\normbig{#1}_1}
\newcommand{\ttnormBig}[1]{\normBig{#1}_1}
\newcommand{\KRopbig}[1]{\mathcal{K}\bigl[#1\bigr]}
\newcommand{\KRopBig}[1]{\mathcal{K}\Big[#1\Big]}
\newcommand{\lbop}{\mathcal{L}}
\newcommand{\opA}{\mathcal{A}}
\newcommand{\dimn}{\mathsf{d}}
\newcommand{\effHami}{H_{\text{eff}}}
\newcommand{\alg}{\opA}
\newcommand{\NA}{\opA_{\dt}}
\newcommand{\UA}{\opA^{(\text{un})}_{\dt}}
\newcommand{\NAARG}[1]{\opA^{(#1)}_{\dt}}

\newcommand{\UAARG}[1]{\opA^{(\text{un},#1)}_{\dt}}
\newcommand{\intterm}{\mathcal{F}}

\newcommand{\krausJ}[1]{\mathcal{J}_{#1}}
\newcommand{\unitdim}[1]{\unit_{#1}}
\newcommand{\bigO}{\mathcal{O}}
\newcommand{\cn}{{Crank–Nicolson}}
\newcommand{\vneq}{{von Neumann equation}}
\newcommand{\pade}{Pad{\'e} approximation}
\newcommand{\padeN}{\mathscr{N}}
\newcommand{\padeR}{\mathscr{R}}
\newcommand{\cube}{\mathfrak{C}}
\graphicspath{{./Fig/}{./assets}{../assets}}
\newcommand{\vect}[1]{\boldsymbol{#1}}
\newcommand{\s}{s}

\newcommand{\revise}[1]{{#1}}

\newcommand{\kw}{Lindblad equation, structure-preserving scheme, error analysis, absolute stability}

\newcommand{\printabstract}{We study a family of structure-preserving deterministic numerical schemes for Lindblad equations. This family of schemes has a simple form and can systemically achieve arbitrary high-order accuracy in theory. Moreover, these schemes can also overcome the non-physical issues that arise from many traditional numerical schemes. Due to their preservation of physical nature, these schemes can be straightforwardly used as backbones for further developing randomized and quantum algorithms in simulating Lindblad equations. In this work, we systematically study \revise{this family of structure-preserving deterministic schemes} and perform a detailed error analysis, which is validated through numerical examples.}

\allowdisplaybreaks

\begin{document}

\title{Structure-preserving Numerical Schemes for Lindblad Equations}

\author{Yu Cao}
\email{yucao@sjtu.edu.cn}
\address{Institute of Natural Sciences and School of Mathematical Sciences, Shanghai Jiao Tong University, Shanghai 200240, China}
\author{Jianfeng Lu}
\email{jianfeng@math.duke.edu}
\address{Department of Mathematics, Department of Physics, and Department of Chemistry, Duke University, Box 90320, Durham, NC, 27708, USA}

\maketitle
\date{}

\begin{abstract}

\printabstract{}

\medskip


\end{abstract}

\section{Introduction}

The study of open quantum systems \cite{breuer_theory_2007}, which describe and characterize the evolution of a quantum system interacting with its environment, is an essential topic in quantum physics, physical chemistry, and information theory. Under the assumption of weak system-environment coupling, the Lindblad equation is a widely used quantum master equation that approximates the dynamical evolution of the system \cite{breuer_theory_2007,davies_markovian_1974}. It is also well-known that the generator of a completely positive (CP) trace-preserving semigroup dynamics must have the Lindbladian form \cite{lindblad_generators_1976,gorini_completely_1976}. 
Lindblad equations have been widely studied in many scientific fields, including but not limited to,  quantum optics \cite{briegel_quantum_1993,carmichael_open_1993}, quantum computation \cite{verstraete_quantum_2009,chi_quantum_2023}, entropy production and thermodynamics \cite{spohn_entropy_1978,alicki_introduction_2018}, superconductivity \cite{kosov_lindblad_2011}.
A recent introductory paper on Lindblad equation could be found in \cite{manzano_short_2020}.

The generic form of Lindblad equation is given as follows \cite{lindblad_generators_1976,gorini_completely_1976}:
\begin{subequations}
\label{eqn::lb}
\begin{align}
\dot{\rho}_t &= \lbop(\rho_t),\\
\lbop(\rho) &:= - i \comm{H}{\rho} + \sum_{k=1}^{\varkappa} \Big(L_k \rho L_k^{\dagger} -\frac{1}{2} \acomm\big{L_k^\dagger L_k}{\rho}\Big), \label{eqn::lbop}
\end{align}
\end{subequations}
where $H$ is Hermitian and is interpreted as the Hamiltonian of the system; $L_k$ are Lindblad operators modeling the interaction with the environment;  $\varkappa\in \Natural$.
When $L_k = 0$ for all $1\le k\le \varkappa$, the above equation reduces to the \vneq{}, which is the density matrix reformulation of Schrödinger equation. Physically, this situation means the system is totally isolated from the environment and hence closed.
Suppose the dimension of the system is $\dimn < \infty$, then the density matrix $\rho_t$ is a $\dimn$-by-$\dimn$ positive semidefinite matrix with unit trace.
The non-Markovian generalization of Lindblad equations has been studied in \eg{}, \cite{breuer_non-Markovian_2007}.

\subsection*{Problem and motivation}

For Lindblad equations, an important numerical problem is to design efficient and effective numerical schemes.
Both deterministic, randomized, and quantum methods have been widely studied for Lindblad equations; see \secref{sec::review} for a comparative review on related deterministic schemes and \secref{sec::discussion} for discussions on randomized and quantum algorithms.
For deterministic schemes, one understudied question is how to numerically simulate Lindblad equations while maintaining the physical structure of density matrices.
It has been recently shown and emphasized by Riesch and Jirauschek \cite{riesch_analyzing_2019}
that classical Runge-Kutta (RK) methods do not preserve the positivity of density matrices even for Hamiltonian evolution ($L_k = 0$ for all $k$ in \eqref{eqn::lb}).
Their work motivates and leads us to present and analyze a family of deterministic structure-preserving schemes for Lindblad equations (see \secref{sec::scheme} below),
based on an idea in the work by Steinbach, Garraway, and Knight \cite{steinbach_high-order_1995}.

\subsection*{Contributions}
Despite a significant amount of works on simulating Lindblad equations, there are very few systematic approaches and presentations to unify deterministic, randomized, and quantum algorithms for this problem. Due to the fundamental physical nature of Lindblad equation, it is not surprising that a natural idea is to use the Kraus operator to approximate the local Lindblad evolution.
In this work, we present a simple family of structure-preserving deterministic schemes that are built based on quadrature schemes for integrals based on \cite{steinbach_high-order_1995}: such a family of positivity-preserving schemes $\alg_{\dt}$ have the Kraus form such that 
 \begin{align*}
\ttnormbig{\rho_T - (\alg_{\dt})^{N}(\rho_0)} \le c\ T^{} \dt^M,
\end{align*}
where the time step $\dt=T/N$, $c$ is some constant, $M$ is the order of the scheme (see Theorem~\ref{thm::error_anal}).
In order to preserve the trace, one can easily normalize $\alg_{\dt}$ to obtain a scheme which both preserves the positivity and unit trace properties in Lindblad equations.
This family of methods are also very suitable and convenient to design other efficient algorithms:
\begin{itemize}[leftmargin=1cm]
\item These schemes can be used as backbones for randomized algorithms by directly randomizing the summation of Kraus operators via any sampling methods, {without the need} to employ numerical schemes for stochastic jump processes or stochastic differential equations, obtained by the unraveling methods \cite{percival_quantum_1998}. Moreover, our framework is more convenient for embracing advanced probabilistic methods like importance sampling \cite{liu_monte_2004} and random batch method \cite{jin_partially_2023}.

\item Any unitary dilation of the Kraus operators (e.g., the Stinespring dilation theorem \cite{stinespring_positive_1955} and its variants) can be applied to these schemes to provide a quantum algorithm for simulating Lindblad equations; see a brief discussion in \secref{sec::discussion}.
\end{itemize}

From the above discussion, it is clear that such a family of Kraus-operator based deterministic schemes are important and useful in developing randomized and quantum algorithms (which we shall elaborate in \secref{sec::discussion}), and these schemes are also efficient in simulating Lindblad equations on a classical computer. As far as we know, there is no previous literature that is fully denoted into presenting such schemes explicitly, analyzing their errors systematically, and discussing their connections with many other algorithms.

\subsection*{Organization of this work}

This paper is organized as follows. In \secref{sec::scheme}, we formulate the idea in \cite{steinbach_high-order_1995} and propose both normalized and unnormalized schemes that preserve the positivity of density matrices, up to any order.
Particularly in \secref{sec::error_anal}, we analyze and quantify the error for  structure-preserving schemes in \thmref{thm::error_anal}. The error analysis is validated by numerical examples in \secref{sec::numerics}. In \secref{sec::discussion}, \revise{we elaborate on how the family of structure-preserving schemes could connect to stochastic and quantum algorithms}. Finally, we provide an ending remark and conclusions in \secref{sec::conclusion}.

\subsection*{Notations}
Suppose $\hbt$ is the Hilbert space of the (open) quantum system, and $L(\hbt)$ is the space of linear operators on $\hbt$. The operators acting on $L(\hbt)$ will be called superoperators, as commonly used in the literature of open quantum systems.
The dimension of $\hbt$ is denoted as $\dimn <\infty$ throughout this paper.
The $\dimn\times \dimn$ identity matrix is denoted by $\unit_\dimn$, and the identity superoperator is denoted by $\id$.
For a two-level quantum system, $\sigma_{X}$, $\sigma_Y$, $\sigma_Z$ represent the Pauli $X$, $Y$, $Z$ matrices, and
$\sigma_{+} = \begin{bsmallmatrix} 0 & 1 \\ 0 & 0 \end{bsmallmatrix}$ and $\sigma_{-}= \begin{bsmallmatrix} 0 & 0 \\ 1 & 0 \end{bsmallmatrix}$.
For any operator $A\in L(\hbt)$, the norm $\norm{A}_{p} := \big(\tr\abs{A}^p\big)^{1/p}$  is the Schatten-$p$ norm.
For any superoperator $\opA: L(\hbt)\rightarrow L(\hbt)$, the norm $\ttnorm{\opA}$  is the induced Schatten-$1$ norm.
The superoperator $\KRopbig{A}$ generated by a Kraus operator $A$ is denoted by
\begin{align*}
\KRopbig{A}(\rho) &:= A \rho A^{\dagger},\qquad \forall \rho,
\end{align*}
which is well-known to be a completely positive superoperator; see e.g., \cite[{Chapter 4.4}]{wilde_quantum_2017}.

\section{On deterministic methods for Lindblad equations}
\label{sec::review}

The solution of Lindblad equation at a particular time $T$ is given by
$\rho_T = e^{\lbop T} (\rho_0)$.
Therefore, this is essentially a computational problem for matrix exponential \cite{moler_nineteen_2003}.
It is a common practice to apply the {scaling and squaring method}, \ie{}, \begin{align*}
e^{\lbop T}(\rho_0) = \big(e^{\lbop \dt}\big)^{T/\dt}(\rho_0).
\end{align*}
This method could reduce the computational problem $e^{\lbop T}(\rho_0)$ into efficiently estimating $e^{\lbop\dt}(\rho)$ for a small time step $\dt$ and a general density matrix $\rho$, which is the main focus below.
In what follows, we will provide a brief review and discussion on existing deterministic numerical methods for simulating (time-independent) Lindblad equations.
More importantly, we shall discuss and focus on the structure-preservation (in particular, the preservation of positivity) for these methods.

In theory, any consistent classical numerical ODE scheme \cite{griffiths_numerical_2010} could simulate the Lindblad equation with theoretical guarantees for a small enough $\dt$.
However, when we take into account additional requirements, \eg{}, the preservation of the physical structure of density matrices, many methods fail to fulfill this requirement, or are computationally expensive for a large dimension $\dimn$.
The preservation of physical properties during the simulation is important especially for a {large-scale system} and {long-run simulation}; see also \cite{riesch_analyzing_2019} for a review on this computational issue. In this part, we will discuss four families of methods: (1) Runge-Kutta and Taylor series methods, (2) \pade{}, (3) operator splitting-based methods, and (4) Kraus representation approximation method. More matrix exponential methods can be found in a thorough review paper by {Moler and Van Loan} \cite{moler_nineteen_2003}. The Kraus representation approximation method will be the guiding principle for methods in \secref{sec::scheme}.

\revise{We remark that apart from the preservation of physical properties, another consideration in algorithms is the ability to deal with \emph{stiff dynamical systems}, for which some implicit methods like Crank–Nicolson method and more generally \pade{} might outperform the Kraus representation approximation method. Such implicit schemes (possibly combined with other techniques) have been discussed and studied in e.g., \cite{bidegaray_introducing_2001,songolo_nonstandard_2018,songolo_strang_2019}. We remark that the discussion here only considers general non-stiff Lindblad equations and stiffness is an aspect beyond the scope of this work.}

\subsection{Runge-Kutta type methods and Taylor series methods}
\label{sec::review::rk}
Explicit Runge-Kutta methods \cite{griffiths_numerical_2010} for Lindblad equations are essentially the same as Taylor expansion methods: for an order $M \ge 1$,
\begin{align}
\label{eqn::rk}
\NAARG{M, \text{RK}}(\rho) := \rho + \sum_{m=1}^{M}\frac{\dt^m}{m!} \lbop^{m}(\rho).
\end{align}
Explicit Runge-Kutta methods could preserve the trace (as $\tr\big(\lbop(\rho)\big)=0$ for any $\rho$), but as was discussed in \cite{riesch_analyzing_2019}, they fail to preserve the positivity of density matrices.

The \cn{} (CN) method, as a second-order implicit Runge-Kutta method,
has also been studied and used for Lindblad equations.
For instance,  \cite{ziolkowski_ultrafast_1995} applied CN scheme for a two-level Lindblad equation with $H$ being time-dependent (in the context of solving the Maxwell-Bloch equation).
When it comes to the \vneq{} $\dot{\rho}_t = -i \comm{H}{\rho_t}$ (a special case of Lindblad equations without dissipative terms), 
it was known that a direct application of \cn{} scheme cannot preserve the positivity in general, when the dimension $\dimn\ge 3$  \cite{bidegaray_introducing_2001}.
In \cite{al-mohy_computing_2011}, Al-Mohy and Higham proposed an efficient algorithm based on the Taylor series method (together with the scaling and squaring method). Their algorithm appears to be very promising and practical. Indeed, the issue of not preserving the density matrix structure might be negligible for many situations, especially when we approximate the matrix exponential accurate enough, \eg{}, up to the machine precision.
However, in the context of simulating Lindblad equation (with physical meaning), it would be better to develop algorithms that respect the physical properties.

\subsection{Pad{\'e} approximation}
The \pade{} is also a widely used method for matrix exponential, especially when it is used in combination with the scaling and squaring method \cite{higham_scaling_2005,al-mohy_new_2009}.
The $(q,q)$-\pade{} for $e^{\lbop \dt}$, denoted by $\padeR_q\bigl(\lbop\dt\bigr)$, is given by
\begin{align}
\label{eqn::pade}
\begin{aligned}
    \padeR_q\bigl(\lbop\dt\bigr) &= \Big(\padeN_q\bigl(-\lbop\dt\bigr)\Big)^{-1} \padeN_q\bigl(\lbop\dt\bigr),\\\padeN_q\bigl(\opA\bigr) &:= \sum_{j=0}^{q} \frac{(2q-j)!q!}{(2q)! j! (q-j)!} \opA^j,
    \end{aligned}
\end{align}
for any operator $\opA$ \cite{moler_nineteen_2003,baker_graves-morris_1996}. As the diagonal $(q,q)$-\pade{} is generally preferred than the general $(p,q)$-\pade{} ($p\neq q$) \cite{moler_nineteen_2003}, we only consider the diagonal ones herein.
When $q = 1$,
$\padeN_1\bigl(\opA\bigr) = \id + \frac{1}{2}\opA.$
Note that $\rho_{k+1} = \padeR_1\big(\lbop\dt\big)(\rho_k)$ in the  $k^{\text{th}}$ iteration is equivalent to
$\padeN_1\big(-\lbop \dt\big) (\rho_{k+1}) = \padeN_1\big(\lbop\dt\big) (\rho_k)$.
This is the same as the \cn{} method mentioned above and thus in general, \pade{} does not preserve the positivity \cite{bidegaray_introducing_2001} in the context of simulating Lindblad equation.

Apart from the fact that \pade{} in general does not respect the physical structure in simulating Lindblad equations, another concern is that the superoperator $\lbop$ has the matrix representation of size $\dimn^2\times \dimn^2$, and computational cost of this method is approximately $\mathcal{O}\big((\dimn^2)^{ 2.81}\big) = \mathcal{O}(\dimn^{5.62})$ for matrix inversion
 with storage space $\mathcal{O}(\dimn^4)$ in general; we do not use the scaling for the theoretically fastest algorithm for matrix multiplication in this discussion, as Strassen-like algorithms \cite{strassen_gaussian_1969} appear to be more practical so far than other theoretically faster algorithms \cite{huang_practical_2018}.
More specifically, notice that from  $\padeN_q\big(-\lbop\dt\big)\rho_{k+1} = \padeN_q\big(\lbop\dt\big)(\rho_k)$ in the  $k^{\text{th}}$ iteration, the right hand side can be evaluated by the action of $\lbop^{m}$ ($1\le m\le q$) to $\rho_{k}$, which only involves matrix multiplication of $\dimn\times \dimn$ matrices. Thus, the cost is $\mathcal{O}(\dimn^{2.81})$ \cite{strassen_gaussian_1969} for the right hand side, and the main concern comes from the matrix inversion on the left.

\subsection{Operator-splitting methods}
\label{subsec::op_split}

For splitting methods, there are two natural choices.
The first one is to split the Lindblad superoperator $\lbop$ according to the Hamiltonian evolution part and the dissipative part as follows:
\begin{subequations}
 \label{eqn::split_HD}
\begin{align}
\lbop &= \lbop_H + \lbop_D, \\
\lbop_H(\rho) := -i \comm{H}{\rho}&, \qquad \lbop_D(\rho) := \sum_{k=1}^{\varkappa} L_k \rho L_k^{\dagger} - \frac{1}{2}\acomm\big{L_k^{\dagger} L_k}{\rho}.
\end{align}
\end{subequations}
{This choice connects to e.g., \cite{bidegaray_introducing_2001,songolo_nonstandard_2018,songolo_strang_2019}. 
We remark that our setup is slightly different from these works and our discussion below only focuses on the mathematical structure. 
The splitting scheme studied in  \cite{bidegaray_introducing_2001,songolo_nonstandard_2018,songolo_strang_2019}
has been used to deal with the stiffness of Maxwell-Liouville-\vneq{}s under certain parameter regions and such a concern is not the focus of this work.}

The second splitting choice is to include some terms inside the dissipative operator $\lbop_D$ into the Hamiltonian $H$, and introduce a notion called {\bf effective Hamiltonian} $\effHami$. Let us define
\begin{align}
\label{eqn::J_and_effHami}
    \effHami := H + \frac{1}{2 i }\sum_{k=1}^{\varkappa} L_k^{\dagger} L_k, \qquad {J} := -i \effHami,
\end{align}
where the notation $J$ is introduced for convenience later.
Then we could rewrite the Lindblad superoperator as
\begin{subequations}
\label{eqn::split_JL}
\begin{align}
&\lbop = \lbop_{J} + \lbop_L,\\
\lbop_{J}(\rho) := {J} \rho  + &\rho {J}^{\dagger},  \qquad \lbop_L(\rho) := \sum_{k=1}^{\varkappa} L_k \rho L_k^{\dagger}.
\end{align}
\end{subequations}

Note that $\lbop_L$ has the Kraus representation form, and is thus a CP superoperator,
whereas $\lbop_D$ is not.
For this reason, in \secref{sec::scheme}, we shall adopt the second splitting choice \eqref{eqn::split_JL},
instead of the first one \eqref{eqn::split_HD};
see \secref{sec::scheme} for details.
Also, it is not hard to verify that $e^{\lbop_J t}$ is a quantum operation, and in particular, a CP superoperator.

\begin{lemma}
\label{lem::evol_exp_lbopJ}
For any $t\in \Real$, the superoperator $e^{\lbop_J t}$ is CP with the Kraus operator $e^{J t}$, \ie{}, for any density matrix $\rho$
\begin{align*}
e^{\lbop_J t}(\rho) = \KRopbig{e^{J t}}(\rho).
\end{align*}
When $t\ge 0$, the superoperator $e^{\lbop_J t}$ is a \revise{(non-trace-preserving) quantum operation}, namely, for any density matrix $\rho$,
\begin{align}
\label{eqn::op_lbopJ_norm}
\ttnorm{e^{\lbop_J t}(\rho)} \le 1.
\end{align}
\end{lemma}

The complete positivity of $\lbop_L$ and $e^{\lbop_J t}$ are essential for the discussion below. We provide a proof of this lemma in \appref{app::exp_J_exp_L} for self-containment.

After splitting the Lindblad superoperator $\lbop$ into two superoperators, we could in principle choose any favorite splitting scheme.
Let us consider, \eg{}, the Strang splitting scheme \cite{strang_construction_1968}, which reduces  approximating $e^{\lbop \dt}$ into calculating matrix exponential of $\lbop_H$ and $\lbop_D$:
\begin{align*}
    e^{\lbop \dt} = e^{\lbop_H \dt/2} e^{\lbop_D \dt} e^{\lbop_H \dt/2} + \order{\dt^3}.
\end{align*}
For the pure Hamiltonian evolution $e^{\lbop_H \dt/2}(\cdot) = e^{-i H \dt/2}(\cdot) e^{i H \dt/2}$, one could use \cn{} \cite{higham_runge-kutta_1996} to approximate $e^{\pm i H \dt/2}$ (at the level of wave function), and this choice would preserve the physical structure of density matrices; more specifically,
$e^{-i H \dt/2} (\cdot) e^{i H \dt/2}  = \big(\unitdim{\dimn} + i H \frac{\dt}{4}\big)^{-1}\ \big(\unitdim{\dimn} - i H \frac{\dt}{4}\big)\ (\cdot)\ \big(\unitdim{\dimn} + i H \frac{\dt}{4}\big)\ \big(\unitdim{\dimn} - i H \frac{\dt}{4}\big)^{-1} + \mathcal{O}(\dt^3)$ \cite{bidegaray_introducing_2001};
non-standard (adaptive) schemes for  two-level Hamiltonian evolution were explored and discussed in \cite{songolo_nonstandard_2018}, and for general-level case in \cite{songolo_strang_2019}.
However, the term $e^{\lbop_D\dt}$ still require additional treatment, especially for large dimensional systems ($\dimn\gg 1$),
where computing matrix exponential might be expensive without knowing particular physical structure of $\lbop_D$.
One could surely use schemes from \secref{sec::scheme} to further approximate $e^{\lbop_D \dt}$, which is another motivation to study structure-preserving schemes in \secref{sec::scheme}.

Next let us consider the second choice in \eqref{eqn::split_JL}. The Strang splitting scheme gives
\begin{align*}
     e^{\lbop \dt} = e^{\lbop_J \dt/2} e^{\lbop_L \dt} e^{\lbop_J \dt/2} + \order{\dt^3}.
\end{align*}
As we have mentioned, $\lbop_L$ is a CP superoperator, thus $e^{\lbop_L \dt}$ is also a CP superoperator. Moreover, $e^{\lbop_L \dt}$ could be approximated further by the finite Taylor series truncation without losing the complete positivity.
As for the term $e^{\lbop_J \dt/2}$, it is also a CP superoperator (see \lemref{lem::evol_exp_lbopJ}), and it could be approximated by an operator in the Kraus form up to any order (see \eqref{eq::approxexpJ} below).
Therefore, for the second splitting choice, we could indeed have a second-order approximation scheme that preserves the positivity and no matrix exponential needs to be involved.

Apart from computational cost, another major concern of splitting-based methods comes from designing higher order schemes.  For example, let us consider a fourth-order scheme. A natural idea is to use Lie-Trotter-Suzuki method \cite{suzuki_general_1991}.
However, due the the unavoidable occurrence of negative time weight for any order larger than or equal to $3$ in Lie-Trotter-Suzuki decomposition  \cite[Theorem 3]{suzuki_general_1991}, some terms like $e^{t\lbop_L}$ (with $t<0$) must occur, and we know that $e^{t\lbop_L}$ with $t<0$ is not even positivity-preserving; see \appref{app::exp_J_exp_L} for an example.
Currently, we are unaware of an effective high-order splitting-based scheme that maintains the complete positivity and at the same time only requires relatively cheap computational cost. There is no doubt that a second-order scheme might already fulfill the need for many examples. Nevertheless, it is still better to have a method that can systematically achieve any high order (at least, up to the order of 3 or 4), without losing the positivity.

\subsection{Kraus representation approximation method}
\label{subsec::kraus_approx}

In general, it is expensive to directly compute matrix exponential of  superoperators,
which we would like to avoid.
Besides, classical Runge-Kutta methods might lead into unphysical results as they do not preserve the positivity in general \cite{riesch_analyzing_2019}. Therefore, we would like to study potentially cheaper approximation methods for $e^{\lbop\dt}$, and these schemes should also preserve the positivity and unit-trace of density matrices.
We notice that a family of potentially promising schemes have been mentioned by Steinbach, Garraway, and Knight in \cite{steinbach_high-order_1995} in the context of studying high-order unraveling schemes.
Even though they began with Runge-Kutta type schemes to derive high-order unraveling methods, they recognized that a series expansion of the splitting choice \eqref{eqn::split_JL} together with quadrature methods (trapezoidal rule was used therein) would also produce the unraveling scheme that they needed.
The connection that they observed can be understood as differential and integral forms to derive numerical schemes, in the context of numerical analysis.
We observe that because the integral approach briefly mentioned in \cite{steinbach_high-order_1995} has Kraus representation form, it could be used as a framework to derive high-order structure-preserving schemes, and matrix exponential could be completely avoided by a further simple approximation \eqref{eq::approxexpJ} below.
We believe this could also be similarly applied to simulate Maxwell-Liouville-\vneq{}s from quantum optics, which we shall leave as future research.
This paper is devoted to studying this Kraus representation approximation method in a more systematic way with detailed numerical analysis, which appears to be missing in literature, to the best of our knowledge.

We shall present how the Kraus representation forms are derived from series expansion in \secref{sec::scheme}.
In the following, we would like to briefly discuss the generic form of Kraus representation approximation method.
In order to preserve the positivity of density matrices, it is natural to consider the Kraus representation as follows:
\begin{align}
\label{eqn::generic_ualg}
\UA(\rho) = \sum_{j=1}^{\mathsf{J}} \KRopbig{A_j(\dt)}(\rho),
\end{align}
where $\bigl\{A_j(\dt)\big\}_{j=1}^{\mathsf{J}}$ is a collection of matrices that possibly depend on $\dt$.
As $\UA$ is a CP superoperator, it preserves the positivity of density matrices (\ie{} $\UA(\rho)$ is positive semidefinite for any density matrix $\rho$),
but it might not preserve the unit trace of density matrices; the superscript \enquote{un} is thus used to indicate \enquote{unnormalized}.

Given any unnormalized scheme in the above form \eqref{eqn::generic_ualg},
it is easy to come up with a normalized scheme $\NA$, defined as
\begin{align}
\label{eqn::normalized_scheme}
    \NA(\rho) := \UA(\rho)/\tr\bigl(\UA(\rho)\bigr),
\end{align}
where we normalize the positive semidefinite matrix $\UA(\rho)$ as a post-processing step. 
{The schemes $\NA$ are convex quasi-linear operators \cite{rembielinski_nonlinear_2020,rembielinski_nonlinear_2021}.}
Note that the normalized scheme $\NA$ is non-linear with respect to $\rho$, and this is the source of improved stability. 
As a remark, the normalization procedure won't change the order of the scheme.
Suppose $\UA$ is an order $M$ scheme, namely,
\begin{align*}
    \rho_{\dt} := e^{\lbop\dt} (\rho_0) = \UA(\rho_0) + \bigO(\dt^{M+1}),
\end{align*}
then $\alg_{\dt}$ is also an order $M$ scheme (\ie{}, $\rho_{\dt} = \NA(\rho_0) +
\bigO(\dt^{M+1})$) (see \lemref{lem::normalization_error} below).
The above discussed deterministic methods are summarized in the following table.

\begin{table}[ht!]
\caption{A summary and comparison of deterministic methods to simulate Lindblad equations. The operator-splitting methods are too expensive and thus are not discussed in this table.}
	\centering\setlength{\extrarowheight}{2pt}
	\centering
	\begin{tabular}{cccccccc}
	\toprule
	\multirowcell{3}{Method category} &	\multirowcell{3}{Algorithm} & \multicolumn{2}{c}{Preservation} & \multirowcell{3}{Local truncation error} \\
	\cline{3-4} \cline{6-7}
	&	& \makecell{Trace} & \makecell{Positivity} & &  \\
	\midrule
	RK-$q$ & see \eqref{eqn::rk} & $\cmark$ & $\xmark$ (by \cite{riesch_analyzing_2019}) & $\order{\norm{\lbop}_{1}^{q+1} \dt^{q+1}}$ \\
	\hline
	$(q,q)$-Pad{\'e} & see \eqref{eqn::pade} & $\cmark$ & $\xmark$ (by  \cite{bidegaray_introducing_2001}) & $\order{\norm{\lbop}_{1}^{2q+1} \dt^{2q+1}}$ \cite{baker_graves-morris_1996} &\\
	\hline
	\multirowcell{1}{This work} & $\NAARG{q}$ & $\cmark$ & $\cmark$ & $\order{\dt^{q+1}}$ see also \thmref{thm::error_anal} \\
	\bottomrule
	\end{tabular}
\end{table}

\section{A systemic approach to structure-preserving schemes}
\label{sec::scheme}

We will present a systematic way to develop arbitrarily high-order schemes (in the form of Kraus representation with normalization constants) that preserve positivity and unit-trace.
This method is based on the integral approach mentioned in \cite{steinbach_high-order_1995}; see also the discussion in \secref{subsec::kraus_approx}.
We shall use the splitting choice as in \eqref{eqn::split_JL}.
For readers' convenience, let us recall some notations from \secref{subsec::op_split}:
\begin{align*}
&\lbop =\lbop_{J} + \lbop_L,\\
\lbop_{J}(\rho) := {J} \rho  + &\rho {J}^{\dagger},  \qquad \lbop_L(\rho) := \sum_{k=1}^{\varkappa} L_k \rho L_k^{\dagger},
\end{align*}
where
\begin{align*}
    \effHami := &H + \frac{1}{2 i }\sum_{k=1}^{\varkappa} L_k^{\dagger} L_k, \qquad {J} := -i \effHami.
\end{align*}
The main idea is to write $e^{\lbop \dt}$ as a series expansion, while maintaining the complete positivity; see \secref{subsec::series} below.
Further natural approximations are employed to avoid directly estimating the matrix exponential $e^{\lbop_J t}$ in \secref{subsec::approx_eLJ}, and to simplify the expressions in the series expansion in \secref{subsec::approx_int}.
The steps in \secref{subsec::series} and \secref{subsec::approx_int} have been mentioned in \cite{steinbach_high-order_1995}, whereas a simple approximation \eqref{eq::approxexpJ} in \secref{subsec::approx_eLJ} is the new ingredient and makes the resulting schemes slightly different from \cite{steinbach_high-order_1995}.
In \secref{sec::error_anal}, we will study detailed error bounds, and in \secref{sec::stability}, we will discuss why structure-preserving schemes enjoy improved absolute stability, compared with classical Runge-kutta methods.

\subsection{Step (\rom{1}): Truncated series expansion based on Duhamel's principle.}
\label{subsec::series}
By Duhamel's principle, viewing $\lbop_L$ as a forcing term, we know that
\begin{align*}
\rho_t \equiv e^{\lbop t}(\rho_0) = e^{\lbop_J t} (\rho_0) + \int_{0}^{t} e^{\lbop_J (t-s)}  (\lbop_L \rho_s)\ \dd s.
\end{align*}
After iterations, we have
\begin{align}
\label{eqn::series_expansion}
\begin{aligned}
& \rho_t = \ e^{\lbop_J t} (\rho_0) \\
& + \sum_{m=1}^{M} \int\limits_{0\le s_1 \le \cdots \le s_m \le t} e^{\lbop_J (t - s_{m})} \lbop_L e^{\lbop_J (s_{m} - s_{m-1})} \lbop_L \cdots e^{\lbop_J (s_2 - s_1)} \lbop_L\ e^{\lbop_J s_1} (\rho_0)\ \dd \vect{s}_{1:m}\\
& +  \int\limits_{0\le s_1 \le \cdots \le s_{M+1} \le t} e^{\lbop_J (t - s_{M+1})}\lbop_L e^{\lbop_J (s_{M+1} - s_{M})}\lbop_L \cdots e^{\lbop_J (s_2- s_1)}\lbop_L (\rho_{s_1})\ \dd \vect{s}_{1:{M+1}}.
\end{aligned}
\end{align}
where we adopted $\dd \vect{s}_{1:m} = \dd s_1 \dd s_2 \cdots \dd s_m$ as a short-hand notation.
The key observation here is that both $e^{\mathcal{L}_J t}$ and $\mathcal{L}_L$ are completely positive superoperators.
If we choose the splitting choice as in \eqref{eqn::split_HD}, we could still have a similar series expansion; however, since $\lbop_D$ is not a completely positive operator (not even preserving positivity), the positivity of density matrices cannot be preserved after the finite truncation.
In what follows, we will discuss how to approximate this superoperator and the integrals further without losing the complete positivity.

\subsection{Step (\rom{2}): Approximate $e^{\lbop_J (t-s)}$ by completely positive operators.}
\label{subsec::approx_eLJ}

The term $e^{\lbop_J (t-s)}$ involves matrix exponential, which we would like to avoid. In fact,
for any order $m \ge 0$, (cf. \lemref{lem::lbopJ_diff} in Appendix)
\begin{align}
\label{eq::approxexpJ}
e^{\lbop_J (t-s)} = \krausJ{m}(t-s) + \bigO((t-s)^{m+1}), \qquad \krausJ{m}(t) := \KRopBig{\sum_{\alpha=0}^{m} \frac{J^\alpha t^\alpha}{\alpha!}}.
\end{align}
Note that $\krausJ{m}(\cdot)$ has the Kraus representation form, and thus is also a completely positive superoperator.
Therefore, the complete positivity is not lost via employing the above approximation \eqref{eq::approxexpJ} into \eqref{eqn::series_expansion}.

If we assume that the step size $\dt \ll 1$,
we can re-write the above series expansion \eqref{eqn::series_expansion} by
\begin{align}
\label{eqn::step_two_approx}
\begin{aligned}
&\begin{aligned}
&\rho_{\dt} =\krausJ{M}(\dt) (\rho_0) \\
& + \sum_{m=1}^{M} \int\limits_{0\le s_1 \le \cdots \le s_m \le \dt} \left(\begin{aligned} &\krausJ{M-m}(\dt-s_{m}) \lbop_L \krausJ{M-m} (s_{m}-s_{m-1}) \lbop_L \cdots \\
&\hspace{2em} \krausJ{M-m}(s_2-s_1)\lbop_L \krausJ{M-m}(s_1) (\rho_0)\end{aligned}\right) \dd \vect{s}_{1:m}\\
& + \bigO(\dt^{M+1})\\
\end{aligned}\\
&\begin{aligned}
&= \krausJ{M}(\dt) (\rho_0) \\
& + \sum_{m=1}^{M} \dt^m \int\limits_{0\le s_1 \le \cdots \le s_m \le 1} \left(\begin{aligned} &\krausJ{M-m}(\dt(1-s_{m})) \lbop_L \krausJ{M-m} (\dt(s_{m}-s_{m-1})) \lbop_L \cdots \\
&\hspace{2em} \krausJ{M-m}(\dt(s_2-s_1))\lbop_L \krausJ{M-m}(\dt s_1) (\rho_0)\end{aligned}\right) \dd \vect{s}_{1:m}\\
& + \bigO(\dt^{M+1})\\
\end{aligned}
\end{aligned}
\end{align}
where we used the change of variables in the second equality.

To further simplify the notation, let us introduce
\begin{align}
\label{eqn::interterm}
\begin{aligned}
&\ \intterm_{m}^{M}(s_m, s_{m-1}, \cdots, s_1) \\
&\ \ := \ \krausJ{M-m}(\dt(1-s_{m})) \lbop_L \krausJ{M-m} (\dt(s_{m}-s_{m-1})) \lbop_L \cdots  \\
&\qquad \qquad \cdots \krausJ{M-m}(\dt(s_2-s_1))\lbop_L \krausJ{M-m}(\dt s_1),
\end{aligned}
\end{align}
which is a composition of multiple superoperators in the form of Kraus representation, and is thus also completely positive.
Then
\begin{align}
\label{eqn::step_two_approx_v2}
&\begin{aligned}
\rho_{\dt} &= \krausJ{M}(\dt) (\rho_0) \\
& + \sum_{m=1}^{M} \dt^m \int_{0\le s_1 \le \cdots \le s_m \le 1} & \intterm_{m}^{M}(s_m, \cdots, s_1) (\rho_0)\ \dd \vect{s}_{1:m} \\
&+  \bigO(\dt^{M+1}).
\end{aligned}
\end{align}
Notice that all terms on the right hand side are in the form of Kraus representation, thanks to the approximation from \eqref{eq::approxexpJ}.

\begin{remark}\
\begin{itemize}[wide]
\item $\intterm_{M}^{M}(s_M, \cdots, s_1) \equiv (\lbop_L)^{M}$ is a constant superoperator, independent of $s_1, \cdots, s_M$.

\item If $m < M$, each term $\intterm_{m}^{M}(s_m, \cdots, s_1)$ is a composition of at most $2m+1$ operators, in particular, there are at most $m+1$ operators having the form $\krausJ{M-m}(\cdot)$, and exactly $m$ of them being $\lbop_L$ operators.

\item
    For the term with order $m$, we could in theory approximate $e^{\lbop_J (t-s)}$ by $\krausJ{\alpha}(t-s)$ with $\alpha \ge M - m$, and the only difference is the tail error. To remove the extra degrees of freedom in the scheme design, we shall simply choose $\alpha = M -m$.
\end{itemize}
\end{remark}

\begin{remark}
The integration above with respect to time variables could be explicitly computed, since $\krausJ{M}$ is a polynomial with respect to the time variable: \eg{}, when $m = 1$,
\begin{align*}
& \int\limits_{0\le s_1 \le 1} \krausJ{M-1}(\dt(1-s_1)) \lbop_L \krausJ{M-1}(\dt s_1) (\rho_0)\ \dd s_1\\
=& \sum_{\alpha_1, \beta_1, \alpha_2, \beta_2 = 0}^{M-1}  \frac{ J^{\alpha_2} \lbop_L\bigl(J^{\alpha_1} \rho_0 (J^{\dagger})^{\beta_1}\bigr) (J^{\dagger})^{\beta_2} (\alpha_1+\beta_1)! (\alpha_2 + \beta_2)!}{\alpha_1 !\beta_1! \alpha_2!\beta_2! (\alpha_1 + \alpha_2 + \beta_1 + \beta_2+1)!}\dt^{\alpha_1+\beta_1+\alpha_2+\beta_2}.
\end{align*}
However, a direct computation via the above expansion would significantly increase the computational complexity.
Therefore, we appeal to further approximation of the integration by quadrature methods with respect to time variables.
\end{remark}

\subsection{Step (\rom{3}): Approximate the integration by quadrature methods.}
\label{subsec::approx_int}

The main idea in this step is to apply any appropriate quadrature methods to approximate the nested integral in  \eqref{eqn::step_two_approx_v2}, and
the resulting approximations are  unnormalized schemes that preserve the positivity.
The only requirement is that the errors from quadrature approximations are small enough so that the order of the scheme is not affected.

There are three important examples: for order $M = 1$, there is no need to use any quadrature approximation, and there is only one unnormalized scheme; when $M = 2$, if we apply the famous trapezoidal rule and midpoint rule, we would end up with two different schemes. We postpone more details to \appref{appendix::quadrature} for the detailed derivation, and we simply summarize the final unnormalized schemes below:
\begin{subequations}
  \label{eqn::sp_alg}
\begin{align}
    & \UAARG{1}(\rho) :=  \KRopbig{\unitdim{\dimn} - i \dt \effHami} (\rho) + \dt \lbop_L (\rho);\label{eqn::sp_alg_1}\\
&\begin{aligned}
\UAARG{2, \text{TR}}(\rho) := &\KRopBig{\unitdim{\dimn} + (-i \effHami) \dt + \frac{(-i \effHami)^2 \dt^2}{2}}  (\rho) \\
&+ \frac{\dt}{2} \KRopbig{\unitdim{\dimn} + (-i \effHami) \dt}\lbop_L (\rho)  \\
& +  \frac{\dt}{2} \lbop_L \KRopbig{\unitdim{\dimn} + (-i \effHami) \dt} (\rho) \\
&+ \frac{\dt^2}{2} \lbop_L \lbop_L (\rho);
\end{aligned}\\
& \begin{aligned}
\UAARG{2, \text{MP}}(\rho) := &\KRopBig{\unitdim{\dimn} + (-i \effHami) \dt + \frac{(-i \effHami)^2 \dt^2}{2}} (\rho) \\
& + \dt \KRopbig{\unitdim{\dimn} + (-i \effHami) \frac{\dt}{2}}\lbop_L \KRopbig{\unitdim{\dimn} + (-i \effHami) \frac{\dt}{2}} (\rho) \\
& + \frac{\dt^2}{2} \lbop_L \lbop_L (\rho).
\end{aligned}\label{eqn::sp_alg_2}
\end{align}
\end{subequations}
The corresponding normalized schemes are denoted by $\NAARG{1}$, $\NAARG{2,\text{TR}}$, and $\NAARG{2,\text{MP}}$;
please refer to \eqref{eqn::normalized_scheme} above for the general form of the normalized scheme. 

For simplicity, more complicated \emph{third-order scheme} can be found in \eqref{eqn::third_order::1} and \eqref{eqn::third_order::2} and a particular \emph{fourth-order} scheme can be found in \eqref{eqn::fourth_order} later in Appendix. 

\begin{remark}[Computational costs]
As one could observe, for structure-preserving schemes like above in \eqref{eqn::sp_alg}, the computational cost is dominated by matrix multiplication of size $\dimn\times \dimn$, and we believe this is perhaps the cheapest way that one could expect at the level of density matrices in general. More specifically, using Strassen algorithm, the computational complexity is  $\bigO(\dimn^{2.81})$ \cite{strassen_gaussian_1969}. Explicit Runge-Kutta schemes discussed in \secref{sec::review::rk} also have the same computational complexity scaling with respect to the dimension $\dimn$, but as discussed above, Runge-Kutta schemes cannot preserve the positivity. To clarify, this is only the theoretical scaling with respect to the dimension $\dimn$, and many other factors could affect the actual simulation cost, \eg{}, the physical model $\lbop$.
\end{remark}

\begin{remark}[General high-order schemes]
We comment on the general case.
For high-dimensional integration, Smolyak algorithm \cite{smolyak_quadrature_1963} is an efficient sparse grid method in many situations.
However, it cannot be directly applied to this particular problem, because if one uses negative weights in the quadrature method, it is not clear whether the positivity of density matrices can still be preserved.
The Monte Carlo methods or Quasi-Monte Carlo methods for the integral terms do not involve negative weights, and thus both methods could always help to preserve the complete positivity. In this work, we will focus on deterministic schemes and thus will not further study the Monte Carlo based quadrature methods, which might be an interesting future research direction.
\end{remark}

\subsection{Error quantifications}
\label{sec::error_anal}

We collect error bounds here for schemes with order up to four. The detailed error analysis for the above three approximation steps is given in Appendix~\ref{app::err_details}.
\begin{theorem}[Error analysis]
\label{thm::error_anal}
For the above schemes, if one picks a quadrature scheme at each level that does not affect the order (more specifically, \eqref{eqn::criterion} below is satisfied for each $m$), and assume that the time step $\dt = \frac{T}{N} \le \frac{1}{\norm{J}_{\infty}}$. Then the total error between the normalized structure-preserving scheme and the exact time propagation is bounded by
\begin{align*}
\norm{\big(\alg_{T/N}^{(M)} \big)^{N}(\rho) - e^{T \lbop}(\rho)}_{1} \le c_M \frac{T^{M+1}}{N^M},
\end{align*}
where 
\begin{align*}
c_M &= \frac{46}{(M+1)!} \big(\norm{J}_{\infty} +\norm{\lbop_L}_{1}\big)^{M+1} \\
& +\sum_{m=1}^{M-1} \frac{4 (M-m)!}{M!} C(M,m) \norm{\lbop_L}_1^{m} \norm{J}_{\infty}^{M-m+1},\\
C(M, m) &= \sum_{
\substack{0\le x_1, x_2, \cdots, x_{2m+2} \le M-m\\  \sum_{j=1}^{2m+2} x_j \ge M-m+1}} 
\frac{1}{\prod_{j=1}^{2m+2} x_j!}.
\end{align*}
\end{theorem}

The expression of $C(M, m)$ can be easily computable given specific $M$, $m\in \Natural$ in practice. When $M = 1$, the summation $\sum_{m=1}^{M-1}$ is defined as zero.

\revise{
\begin{remark}
In the initial preprint release, we used the midpoint scheme to approximate the nested integral in \eqref{eqn::step_two_approx_v2}. The framework discussed above was later adopted and analyzed by Li and Wang in \cite{li_simulating_2023} who used scaled Gaussian quadrature, and their resulting scheme is expected to be more efficient than the midpoint scheme. We remark that in the above upper bound in \thmref{thm::error_anal}, we do not assume to use a particular quadrature scheme (as long as it is consistent with \eqref{eqn::criterion} below to maintain the order condition), and this is a different estimate compared with \cite{li_simulating_2023}.
\end{remark}
}

\subsection{Discussion on absolute stability}
\label{sec::stability}

The above developed structure-preserving schemes will demonstrate improved absolute stability compared with the Runge-Kutta type methods. For classical Runge-Kutta, it is well-known that when $\dt$ is relatively large, during iterations $\rho_{k+1} = \alg_{\dt}^{(M,\text{RK})}(\rho_k)$ might diverge to infinity as $k\to\infty$. The reason is that a large $\dt$ changes the spectrum behavior of the operator $\alg_{\dt}^{(M,\text{RK})}$.
However, this is \emph{never} an issue for the normalized structure-preserving scheme above, as $\alg_{\dt}^{(M)}(\cdot)$ maps any density matrix to another one in a non-linear way, which renders such a divergent behavior impossible. This will be demonstrated below in numerical examples in \secref{sec::numerics}.

\section{Numerical examples}
\label{sec::numerics}

We demonstrate the performance of structure-preserving schemes developed in \secref{sec::scheme} for three examples: a two-level decaying Lindblad equation ($\dimn = 2$), a two-level atom interacting with quantized photon field ($\dimn = 4$, $10$ and $20$ are considered), and a 1D dissipative Ising model with $2$, $4$, or $6$ atomic sites (namely, $\dimn=4$, $16$, $64$ respectively).
We will demonstrate the order of convergence for structure-preserving schemes, as well as their improved stability for large $\dt$.
We shall refer our structure-preserving (SP) schemes as $\text{SP}k$ where $k$ is the order of the scheme. The detailed expressions of structure-preserving schemes being tested below can be found in \eqref{eqn::sp_alg_1} (1st order), \eqref{eqn::sp_alg_2} (2nd order, midpoint), \eqref{eqn::third_order::1} (3rd order), \eqref{eqn::fourth_order} (4th order).

\subsection{A two-level decaying system}

We first consider a two-level Lindblad equation from \cite[Eq. (3.219)]{breuer_theory_2007} with the following choice
\begin{align}
\label{eqn:lb-decay-2}
H = 0, \qquad
L_1 = \sqrt{\lambda_0 (\nu+1)} \sigma_{-}, \qquad
L_2 = \sqrt{\lambda_0 \nu} \sigma_{+},
\end{align}
where $\lambda_0$ is the spontaneous emission rate, and $\nu$ is the value of Planck distribution at the transition frequency \cite{breuer_theory_2007}.

In \figref{fig:order-convergence}, we visualize the averaged terminal error
at time $T = 1$ for various schemes,
where $N = T/\dt$ is the number of time steps,
and the expectation with respect to $\rho_0$ is approximated by $5$ randomly
generated samples via \textsf{QuTiP} \cite{Johansson_2013_qutip}.
The convergence behavior for structure-preserving schemes could be clearly observed in \figref{fig:order-convergence}.
For some parameters (e.g., $\lambda_0 = 3$), some error curves for SP schemes appear to lie above the error curves for Runge-Kutta with the same order, namely, the prefactor in front of the error scaling of SP schemes might be larger than that of Runge-Kutta schemes.
This phenomenon is likely to originate from the expansion in \eqref{eqn::series_expansion}, whose truncation error is smaller when the interaction strength $\norm{\lbop_L}$ is small (namely, the weak-coupling region). However, for many physically relevant Lindblad equations, weak-coupling assumption is expected \cite{breuer_theory_2007}, and in this physical region, SP schemes appear to have a smaller prefactor in error scaling compared with Runge-Kutta (cf. the case $\lambda_0 = 1$ in \figref{fig:order-convergence}).

In \figref{fig:decay-2-stability}, we test the stability of various schemes for $\lambda_0 = 5$ and $\nu = 1/2$.
The initial density matrix is chosen as $\rho_0 = \frac{1}{2}\big(\unit_2
+ \frac{1}{\sqrt{6}}\sigma_X + \frac{1}{\sqrt{3}}\sigma_Y
+ \frac{1}{\sqrt{2}}\sigma_Z\big)$, and we visualize
$\abs{\langle\sigma_{X}\rangle_{\rho_t}}$ and $\abs{\langle\sigma_{Y}\rangle_{\rho_t}}$
with respect to time $t$,
where $\rho_t$ is approximated by different schemes.
As one could observe, for structure-preserving schemes, both
$\abs{\langle\sigma_{X}\rangle_{\rho_t}}$ and
$\abs{\langle\sigma_{Y}\rangle_{\rho_t}}$ always decay for the time step chosen therein; for Runge-Kutta scheme $\NAARG{2,\text{RK}}$, the asymptotic decay
behavior is only preserved for small $\dt<0.4$.
Therefore, we have verified that structure-preserving schemes \eqref{eqn::sp_alg} have better absolute stability, compared to classical ODE solvers such as $\NAARG{2,\text{RK}}$.

\begin{figure}[h]
\centering
\includegraphics[width=0.9\textwidth]{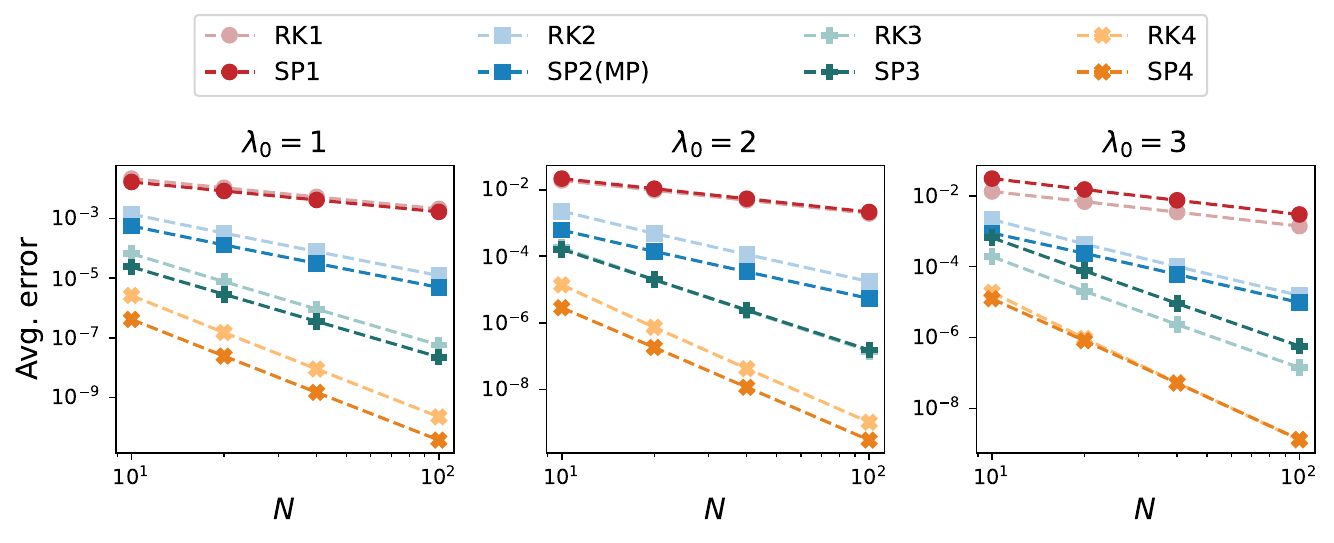}
\caption{{\bf (A two-level system).} We visualize the averaged terminal error $\ee_{\rho_0} \bigl[\ttnorm{(\NA)^{N}(\rho_0) - e^{\lbop T}(\rho_0)}\bigr]$ for $T = 1$. The Lindblad equation $\rho_t$ under consideration is a decaying two level system \eqref{eqn:lb-decay-2} with varying $\lambda_0$ while $\nu =1/2$ is fixed.}
\label{fig:order-convergence}
\end{figure}

\begin{figure}[htbp]
\centering
\includegraphics[width=0.65\textwidth]{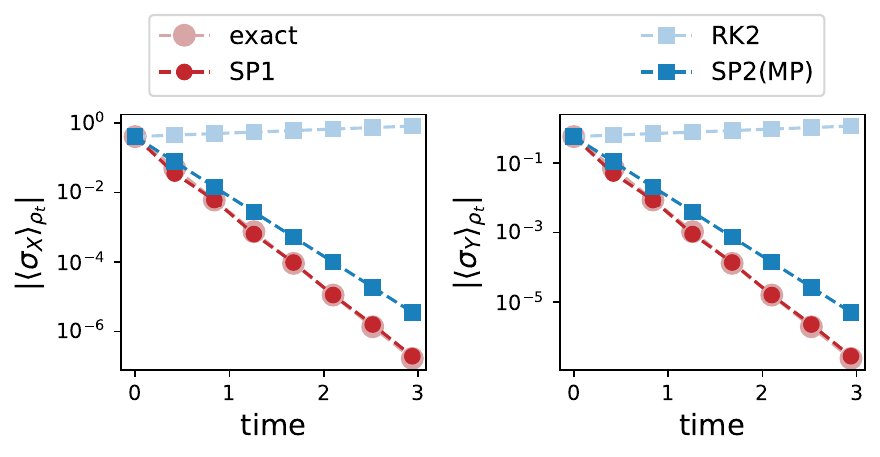}
\caption{{\bf (Stability demonstration in a two-level system).} This figure shows $\abs{\langle \sigma_{X}\rangle_{\rho_t}}$
and $\abs{\langle \sigma_{Y}\rangle_{\rho_t}}$
with respect to time $t$, and $\rho_t$ is approximated via various
schemes with {\bf a large time step} $\dt = 0.42$.
The Lindblad equation $\rho_t$ under consideration is a decaying two level system \eqref{eqn:lb-decay-2} with $\lambda_0=5$ and $\nu =1/2$, and with initial condition $\rho_0 = \frac{1}{2}\big(\unit_2
+ \frac{1}{\sqrt{6}}\sigma_X + \frac{1}{\sqrt{3}}\sigma_Y
+ \frac{1}{\sqrt{2}}\sigma_Z\big)$.
}
\label{fig:decay-2-stability}
\end{figure}

\subsection{A two-level atom interacting with a quantized photon field}
We further consider a Lindblad equation studied in \cite{briegel_quantum_1993},
for a composite system consisting of an two-level atom and a quantized photon field.
The Hamiltonian term in the Lindblad equation is given by
\begin{align}
   \label{eq:atom-photon-hami}
    H = \id_{\text{atom}}\otimes \bigl(\omega a^{\dagger} a\bigr) + \bigl(\Omega
    \sigma_Z\bigr)\otimes \id_{\text{ph}} - g (\sigma_{-}\otimes a^{\dagger}
    + \sigma_{+}\otimes a),
\end{align}
where $a^{(\dagger)}$ are annihilation/creation operators for the photon field,
Pauli matrices act on the two-level atom,
and $\id_{\text{atom}}$ and $\id_{\text{ph}}$ are identity superoperators acting on the atom and the photon field, respectively;
$\omega$, $\Omega$, $g$ are parameters, and in particular,
$g$ measures the interaction strength of the atom and the photon field,
and is known as the {Rabi frequency}. Lindblad operators are given by
\begin{align}
  \label{eq:atom-photon-Lindbladop}
    \begin{aligned}
        L_1 &= \id_{\text{atom}} \otimes \bigl(\sqrt{\alpha(\nu+1)}\ a\bigr),\qquad
        L_2 = \id_{\text{atom}}\otimes \bigl(\sqrt{\alpha\nu}\ a^{\dagger}\bigr),\\
        L_3 &= \bigl(\sqrt{\beta(1-\eta)}\ \sigma_{-}\bigr) \otimes \id_{\text{ph}},
    \qquad
        L_4 = \bigl(\sqrt{\beta \eta}\ \sigma_{+}\bigr) \otimes \id_{\text{ph}},\\
        L_5 &= \bigl(\sqrt{\gamma}\ \sigma_Z\bigr) \otimes \id_{\text{ph}},
    \end{aligned}
\end{align}
where $\alpha$, $\beta, \gamma$, $\nu$ are non-negative constants, and
the parameter $\eta\in[0,1]$.
For our purpose of numerical experiment, we set the  parameters as $\omega = \Omega = g = 1$, $\nu = \eta = 1/2$, and choose  $\alpha = \beta  = \gamma$.
We truncate the dimension of the photon field to $2, 5, 10$
(therefore, the dimension of the composite quantum system is $\dimn=4$, $10$, $20$ respectively).
In \figref{fig::conv_atom_photon}, we present the averaged terminal error $\ee_{\rho_0} \bigl[\ttnorm{(\NA)^{N}(\rho_0) - e^{\lbop T}(\rho_0)} \bigr]$ with respect to $N$ for various schemes $\NA$
and for several sets of parameters.
The expectation $\ee_{\rho_0}[\cdot]$ is approximated by
$5$ randomly generated density matrices $\rho_0$ in tensor product form (namely, the atom and the photon field are initially not interacting with each other).
In \figref{fig::conv_atom_photon}, the orders of convergence for various SP schemes are validated. We can observe a similar phenomenon that SP schemes have relatively better error scaling prefactor when the interaction strength is smaller.

\begin{figure}
\centering
\begin{subfigure}[b]{0.9\textwidth}
    \includegraphics[width=\textwidth]{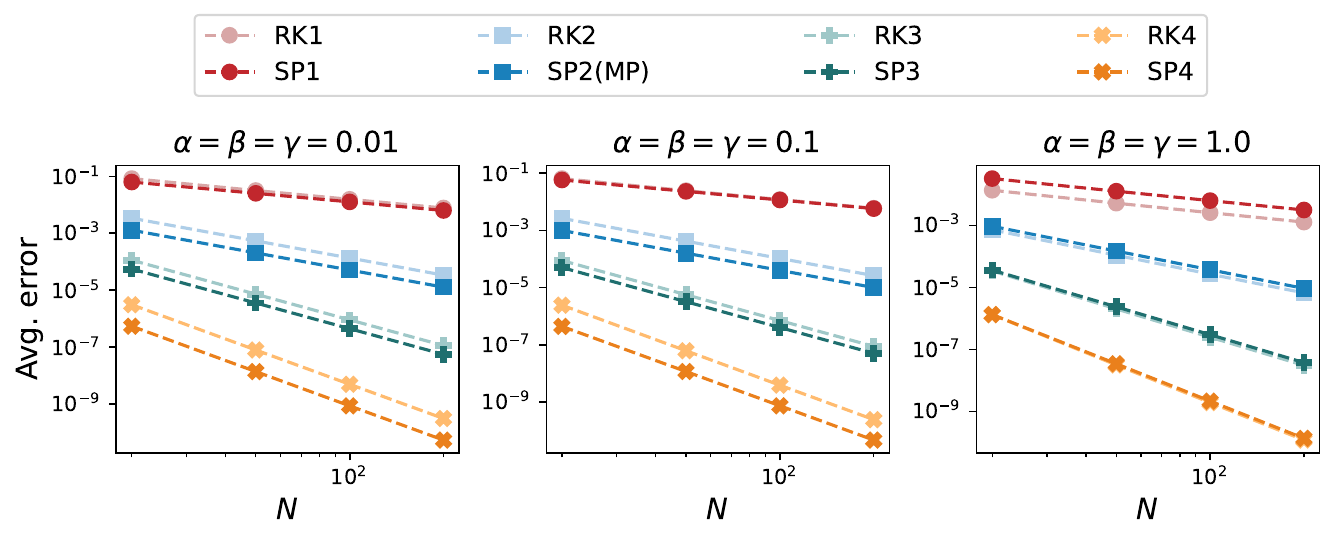}
    \caption{number of photon level is $2$ ($\dimn=4$)}
\end{subfigure}
\\
\begin{subfigure}[b]{0.9\textwidth}
    \includegraphics[width=\textwidth]{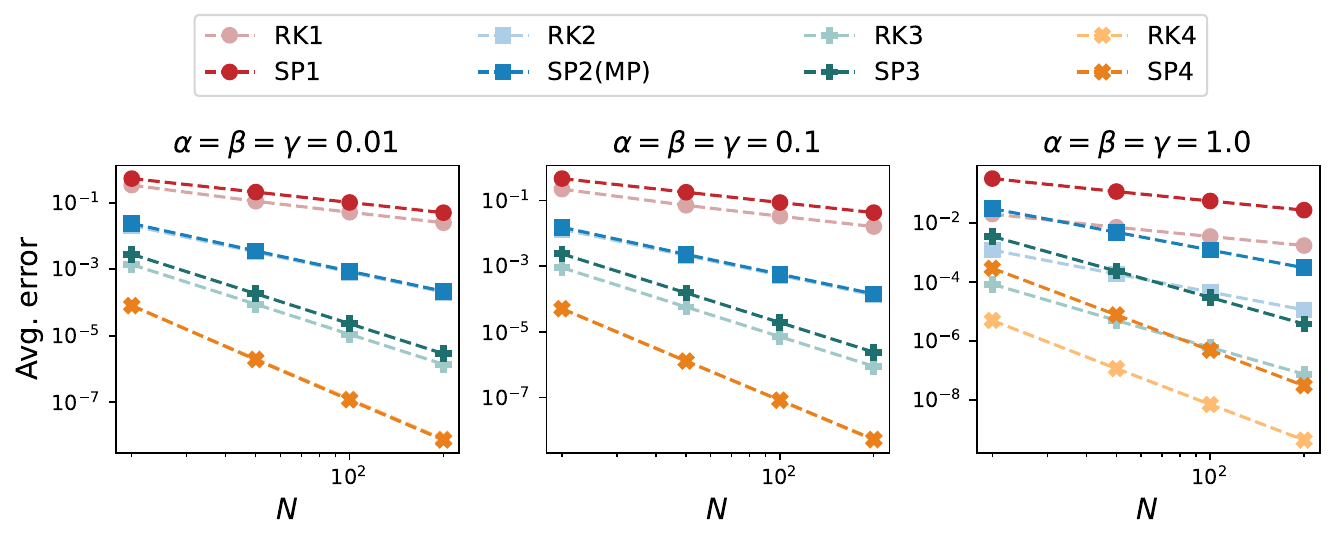}
    \caption{number of photon level is $5$ ($\dimn=10$)}
\end{subfigure}
\\
\begin{subfigure}[b]{0.9\textwidth}
    \includegraphics[width=\textwidth]{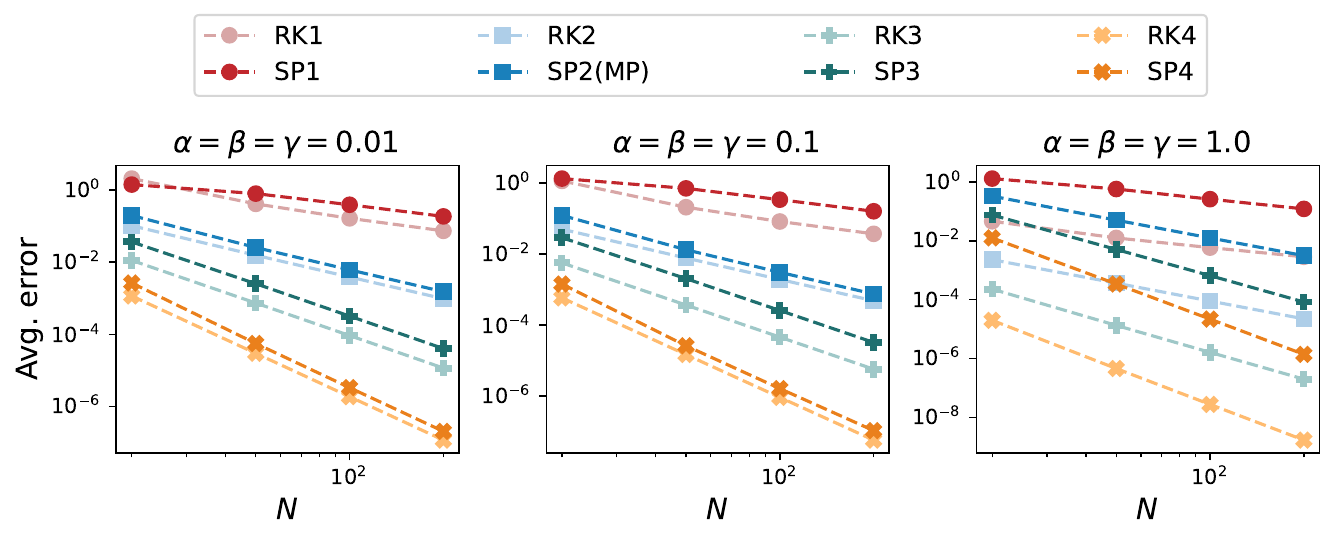}
    \caption{number of photon level is $10$ ($\dimn=20$)}
\end{subfigure}
\caption{{\bf (Atom interacting with photon).} We visualize averaged terminal error $\ee_{\rho_0} \bigl[\ttnorm{(\NA)^{N}(\rho_0) - e^{\lbop T}(\rho_0)}\bigr]$ for $T = 1$. The Lindblad equation herein is given by \eqref{eq:atom-photon-hami} and \eqref{eq:atom-photon-Lindbladop}.}
\label{fig::conv_atom_photon}
\end{figure}

\subsection{Dissipative Ising model}
Finally, we shall consider the following 1D dissipative Ising model for $n$ spins, with only nearest neighbor interaction in Hamiltonian:
\begin{align}
\label{eqn::ising}
H = \sum_{i=1}^n \sigma_Z^{(i)} - \sum_{i=1}^{n-1} \sigma_X^{(i)}\otimes \sigma_X^{(i+1)}, \qquad L_i = \sqrt{\gamma} \sigma_{-}^{(i)}\ \text{ for } i = 1, 2, \cdots, n,
\end{align}
where $\gamma > 0$ characterizes the interaction strength of the spin system with the environment; the superscript $\sigma_Z^{(i)}$ in Pauli-Z means that the Pauli-Z matrix acts on the site $i$, and similar notations apply to other Pauli matrices. For simplicity, we shall consider the terminal time $T=1$ in numerical experiments, and consider $5$ randomly generated product states as initial conditions to ensure certain robustness. In Figure~\ref{fig::conv_ising}, we can easily observe the order of convergence for SP schemes by comparing them with RK schemes, and moreover, SP schemes appear to have smaller error prefactor compared with RK schemes when $\gamma$ is not too large.

\begin{figure}
\centering
\begin{subfigure}[b]{0.9\textwidth}
    \includegraphics[width=\textwidth]{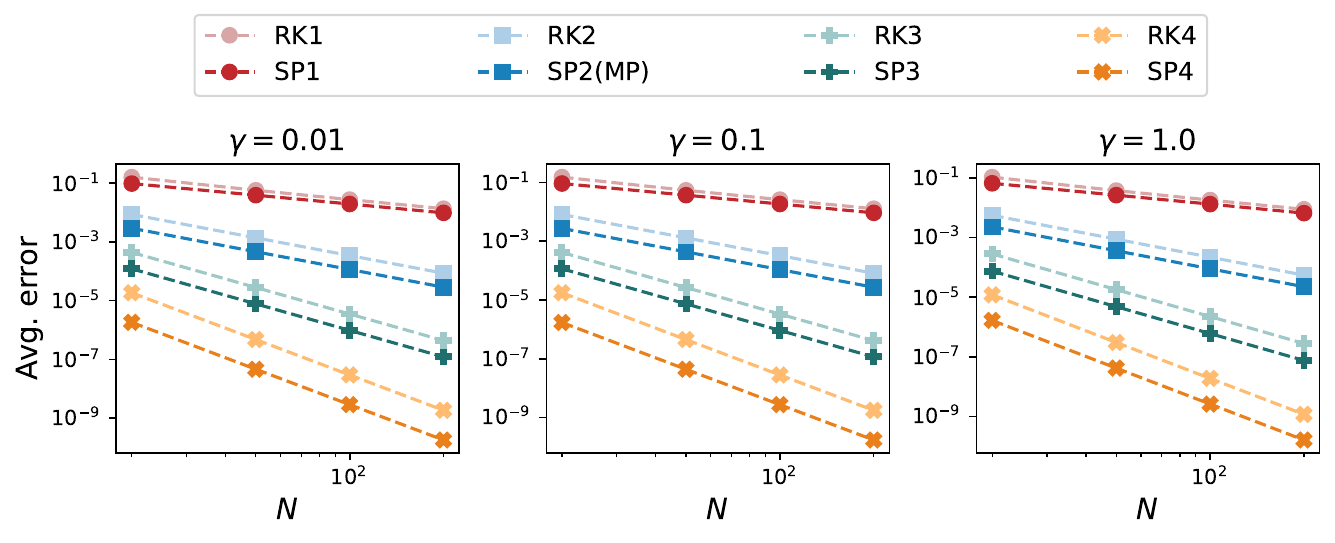}
    \caption{$n = 2$ ($\dimn=4$)}
\end{subfigure}
\\
\begin{subfigure}[b]{0.9\textwidth}
    \includegraphics[width=\textwidth]{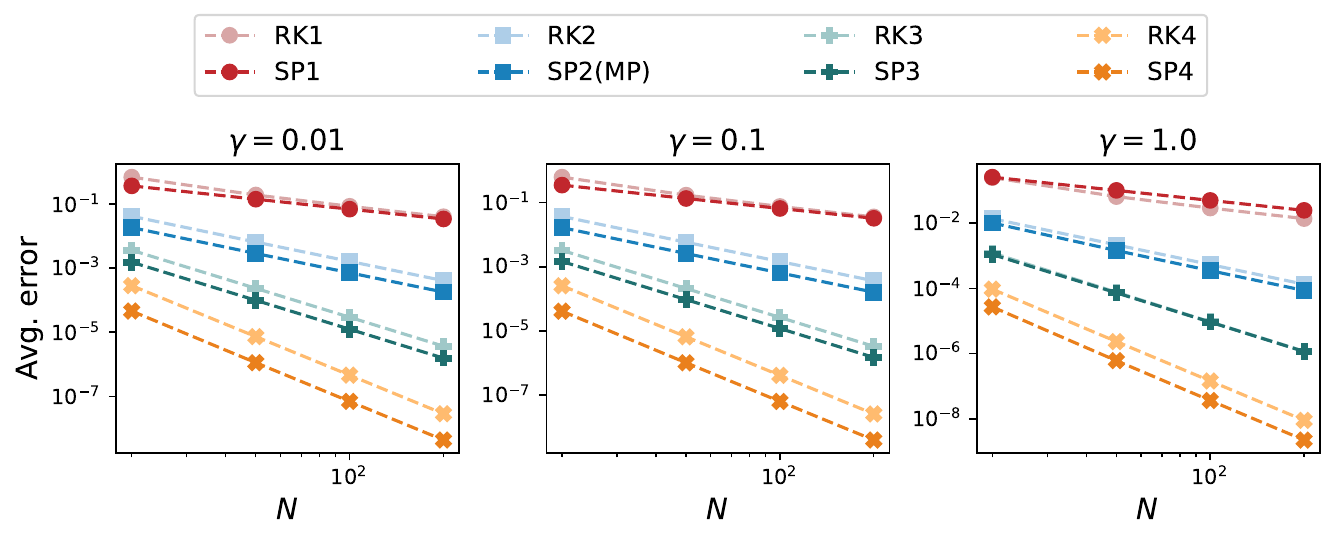}
    \caption{$n=4$ ($\dimn=16$)}
\end{subfigure}
\\
\begin{subfigure}[b]{0.9\textwidth}
    \includegraphics[width=\textwidth]{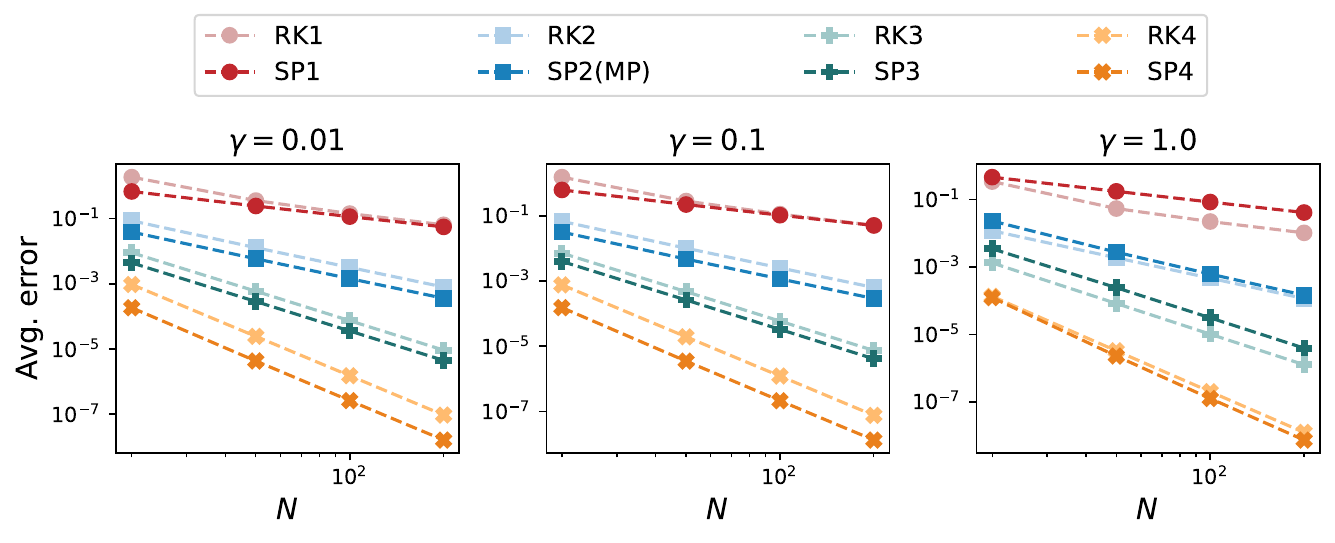}
    \caption{$n=6$ ($\dimn=64$)}
\end{subfigure}
\caption{{\bf (1D dissipative Ising model).} We visualize averaged terminal error $\ee_{\rho_0} \bigl[\ttnorm{(\NA)^{N}(\rho_0) - e^{\lbop T}(\rho_0)}\bigr]$ for $T = 1$. The Lindblad equation herein is given by \eqref{eqn::ising}.}
\label{fig::conv_ising}
\end{figure}

\section{Discussion on randomized and quantum algorithms}
\label{sec::discussion}

Due to its apparent preservation of physical nature, the structure-preserving algorithms developed above can be used as a backbone for designing randomized algorithms (as known as unraveling methods in literatures) and quantum algorithms for simulating Lindblad equations. In what follows, we shall discuss on their connections, as well as the values that our perspective can possibly bring.

\subsection{Randomized algorithms}

Consider any deterministic unnormalized scheme $\UA$ in Kraus representation \eqref{eqn::generic_ualg}.
We could immediately obtain a corresponding unraveling scheme:
given any wave function $\ket{\psi_0}$ at time $0$, let $\ket{\psi_0}$ jump to a unnormalized wave function
\begin{align}
\label{eqn::unraveling}
    \ket{\psi_{\dt}} = \frac{1}{\sqrt{p_j}}A_j(\dt) \ket{\psi_0}, \qquad \text{ with probability } p_j,
\end{align}
where $p_j >0$ for any index $1\le j\le \mathsf{J}$.
Then suppose at time $0$, $\ket{\psi_0}$ is a random variable. Then
\begin{align*}
    \ee\ \bigl[ \ket{\psi_{\dt}}\bra{\psi_{\dt}}\bigl] = \ee\ \Big[\sum_{j=1}^{\mathsf{J}} A_j(\dt) \ket{\psi_0}\bra{\psi_0} A_j^{\dagger}(\dt)\Big] = \UA\ \Big(\ee\ \bigl[\ket{\psi_0}\bra{\psi_0}\bigr]\Big).
\end{align*}

It is clear that this stochastic process $\ket{\psi_{k\dt}}$ (integer $k\ge 1$) above is a stochastic realization of the unnormalized scheme $\UA$.
The major benefit of stochastic unraveling method \cite{gisin_quantum-state_1992,dalibard_wave-function_1992,percival_quantum_1998} is that it only simulates a wave function $\ket{\psi_{k\dt}}$ with $\bigO(\dimn)$ storage space, at the cost of many samples in order to recover the density matrix.
High-order unraveling schemes have been used to study \eg{}, photodesorbing \cite{andrianov_performance_2003} and dissipative molecular system coupled with external fields \cite{nakano_monte_2003,nakano_second-order_2004,nakano_monte_2003_1}; a time-step adaptive method for the quantum jump unraveling scheme was studied in \cite{kornyik_monte_2019}. 

Our perspective also has values in providing the flexibility to employ other sampling schemes. 
If we start with the traditional unraveling scheme, namely, transforming the Lindblad equation into a  statistically equivalent stochastic differential equations (or jump process), designing a high-order scheme for stochastic differential equations is more complex than applying stochastic schemes directly for Kraus operator sum representation as discussed above. Moreover, the above stochastic algorithm \eqref{eqn::unraveling} based on the structure-preserving algorithms is easier to adopt advanced sampling schemes (e.g., importance sampling, or random batch method \cite{jin_partially_2023}) to achieve further efficiency.

\subsection{Quantum algorithms}
\label{sec::qa}

\revise{Next, we will discuss the connection of the scheme developed in \secref{sec::scheme} with quantum algorithms.
An essential motivation for quantum computers and quantum algorithms is to  simulate quantum dynamics, a field initiated by Feynman in 1982 \cite{feynman_simulating_1982}. Lindblad equation is surely a very important family of quantum dynamics that could benefit from quantum algorithms, see e.g., some recent studies in  \cite{cleve_2017,hu_quantum_2020,schlimgen_quantum_2021,schlimgen_quantum_2022,li_simulating_2023,ding_simulating_2023,watad_variational_2023}.
The above developed Kraus operator decomposition in \secref{sec::scheme} for Lindblad equations provides a systemic approach to achieve high-order approximations of Lindblad equation while maintaining the positivity, which renders this family of schemes suitable as the backbone for developing quantum algorithms. After the initial preprint release of this work, the method developed in \secref{sec::scheme} has already been adopted by Li and Wang in \cite{li_simulating_2023} to develop quantum algorithms for simulating Lindblad equations.}

\section{Conclusion and outlook}
\label{sec::conclusion}

In this work, we have studied a family of structure-preserving schemes for Lindblad equations with detailed error analysis. These schemes are guaranteed to possess improved absolute stability. The validity and performance of these schemes are demonstrated via a two-level decaying system, an atom-photon coupling system,  and 1D dissipative Ising model.

There are a few open questions that are worthwhile to be further explored:
(1) it could also be interesting to design similar structure-preserving schemes with smaller error prefactor;
(2) in this work, we have not incorporated the dynamical low-rank approximation \cite{Lubich07,LeBris13,LeBris15,cao_stochastic_2018}
to further reduce the computational cost whenever applicable,
and a hybrid method by including the above structure-preserving schemes and dynamical low-rank approximation (with adaptive rank) could be promising in practice.

\bibliography{reference.bib}

\appendix

\section{Proofs for \secref{sec::review}}
\label{app::exp_J_exp_L}

We shall provide a proof for \lemref{lem::evol_exp_lbopJ} that $e^{\lbop_J t}$ ($t\ge 0$) is a quantum operation, and also provide an example that $e^{\lbop_L t}$ ($t<0$) is not even positivity-preserving.

\subsection*{Proof of \lemref{lem::evol_exp_lbopJ}}

Suppose $\partial_t \ket{\psi_t} = J \ket{\psi_t}$ with initial condition $\ket{\psi_0}$. We can observe that
\begin{align*}
e^{J t} \ket{\psi_0} \bra{\psi_0} e^{J^{\dagger} t} = \ket{\psi_t} \bra{\psi_t} =: \eta_t.
\end{align*}
It is not hard to verify that 
\begin{align*}
\partial_t \eta_t = \lbop_J (\eta_t), \text{which implies that } \eta_t = e^{\lbop_J t}(\ket{\psi_0}\bra{\psi_0}).
\end{align*}
Therefore,
\begin{align*}
 e^{\lbop_J t}(\ket{\psi_0}\bra{\psi_0}) = e^{J t} \ket{\psi_0} \bra{\psi_0} e^{J^{\dagger} t}.
\end{align*}
Hence, $e^{\lbop_J t}$ preserves the positivity for pure states, and thus also preserves the positivity for general density matrices by the linearity of the operator $e^{\lbop_J t}$;
more specifically,
\begin{align*}
e^{\lbop_J t}(\rho) = \KRopbig{e^{J t}}(\rho), \qquad \forall \rho.
\end{align*}

Next when $t\ge 0$, let us compute the change of the norm $\norm{\ket{\psi_t}}^2$,
\begin{align*}
\partial_t \braket{\psi_t}{\psi_t}
= \bra{\psi_t} J+J^{\dagger}\ket{\psi_t}
= -\bra{\psi_t} \sum_{k=1}^{\varkappa} L_k^\dagger L_k \ket{\psi_t} \le 0.
\end{align*}
Therefore, $0 \le \tr\big(e^{\lbop_J t}(\ket{\psi_0}\bra{\psi_0})\big) \le \tr\bigl(\ket{\psi_0}\bra{\psi_0}\bigr)$. 
The final conclusion easily follows from the linearity.

\subsection*{The superoperator $e^{\lbop_L t}$ ($t<0$) does not preserve the positivity}
Let us consider an example $\lbop_L(\cdot) = \sigma_X (\cdot) \sigma_X$.
Consider any density matrix $\rho = \frac{1}{2} \bigl(\unit_2 + r_X \sigma_X + r_Y \sigma_Y + r_Z \sigma_Z\bigr)$, where $r_X^2+r_Y^2+r_Z^2\le 1$.
One could verify that
\begin{align*}
    \lbop_L(\rho) = \frac{1}{2}\bigl(\unit_2 + r_X \sigma_X - r_Y \sigma_Y - r_Z \sigma_Z\bigr),
\end{align*}
and thus
\begin{align*}
    e^{t\lbop_L}(\rho) &= \frac{1}{2}\bigl(e^t\unit_2 + e^t r_X \sigma_X + e^{-t} r_Y \sigma_Y + e^{-t}r_Z \sigma_Z\bigr) \\
    &= e^t\frac{1}{2} \bigl(\unit_2 + r_X \sigma_X + e^{-2t}r_Y \sigma_Y + e^{-2t}r_Z \sigma_Z\bigr).
\end{align*}
To show that the matrix $e^{t\lbop_L}(\rho)$ is not positive semidefinite in general, notice that the prefactor $e^t$ does not matter, and inside, it has a Bloch vector form with coefficients $(r_X, e^{-2t} r_Y, e^{-2t}r_Z)$. Therefore, it is a positive-semidefinite matrix if and only if
$(r_X)^2 + \bigl(e^{-2t} r_Y\bigr)^2 + \bigl(e^{-2t} r_Z\bigr)^2 \le 1.$
For any $t<0$, obviously, the above relation does not hold in general for all valid $(r_X, r_Y, r_Z)$ with $r_X^2+r_Y^2+r_Z^2\le 1$. Therefore, $e^{t\lbop_L}$ might map a density matrix to a matrix with a negative eigenvalue, and thus $e^{t\lbop_L}$ ($t<0$) is not positivity-preserving, let alone completely positivity.

\section{On the optimal quadrature scheme and error analysis}
\label{appendix::quadrature}

\subsection{A general principle for quadrature approximation}

As discussed above, when $m = M$, there is no need to consider any quadrature scheme. Therefore, when $M = 1$, the only scheme is that 
\begin{align}
\label{eqn::order_one}
\rho_{\dt} =  \krausJ{1}(\dt)(\rho_0) + \dt \lbop_L(\rho_0) + \order{\dt^2}.
\end{align}

Next, let us consider a general $M>1$ and consider an arbitrary $m$ with $1\le m < M$. 
Let us denote the hyper-cube (also known as the probabilistic simplex) as $\cube_m$ for convenience: 
\begin{align*}
\cube_m := \Big\{(\s_1, \s_2, \cdots, \s_m)\ |\ 0 \le \s_1 \le \s_2 \cdots \le \s_m \le 1 \Big\}.
\end{align*}
 Then for each integral term inside \eqref{eqn::step_two_approx_v2}, 
\begin{align*}
&\int_{\cube_m} \intterm_{m}^{M}(s_{m}, \cdots, s_1) (\rho_0)\ \dd \vect{s}_{1:m}\\
=& \sum_{\substack{0\le \alpha_0, \alpha_1,\cdots,\alpha_m,\\ \beta_0, \beta_1,\cdots, \beta_m \le M-m}} \Big(\int_{\cube_m} \frac{\prod_{j=0}^{m} \dt^{\alpha_j+\beta_j}(\s_{j+1}-\s_j)^{\alpha_j+\beta_j}}{\prod_{j=0}^{m}\alpha_j! \prod_{j=0}^{m}\beta_j!}\ \dd \vect{s}_{1:m}\Big) \text{Op}_{\vect{\alpha},\vect{\beta}},
\end{align*}
where
\begin{align*}
\text{Op}_{\vect{\alpha},\vect{\beta}} := \mathfrak{J}_{\alpha_m,\beta_m}\circ \lbop_L \circ \mathfrak{J}_{\alpha_{m-1},\beta_{m-1}}\circ  \cdots \circ \lbop_L \circ \mathfrak{J}_{\alpha_0,\beta_0} (\rho_0),\qquad \mathfrak{J}_{\alpha,\beta}(\cdot) := J^\alpha (\cdot) (J^\dagger)^{\beta},
\end{align*}
and we introduced $\s_{m+1}\equiv 1$ and $\s_0\equiv 0$ for simplicity. For each index $0\le j\le m$, $\alpha_j, \beta_j$ are integers bounded by $0$ and $M-m$, 
\begin{align*}
\vect{\alpha} = (\alpha_0, \alpha_1, \cdots, \alpha_m) \in \Natural^{m+1},\qquad \vect{\beta} = (\beta_0, \beta_1, \cdots, \beta_m) \in \Natural^{m+1}.
\end{align*}
For later convenience, let us define
\begin{align*}
\vect{\gamma} = (\gamma_0, \gamma_1, \cdots, \gamma_m) := \vect{\alpha} + \vect{\beta} = (\alpha_0 + \beta_0, \cdots, \alpha_{m}+\beta_m).
\end{align*}

We can easily observe that:
\begin{lemma}
\label{eqn::op_norm}
The norm of $\norm{\text{Op}_{\vect{\alpha},\vect{\beta}}}_{1} \le \norm{J}_{\infty}^{\sum_{j=0}^m \alpha_j + \beta_j} \norm{\lbop_L}^{m}_1$.
\end{lemma}

A simple application of the triangle inequality leads into the following error quantifications:
\begin{proposition}[A general computable error bound]
\label{prop::error_1}
Suppose we find a quadrature scheme with positive weights $\vect{w}=(w_1, w_2, \cdots, w_Q)$ and $Q$-samples $\vect{\s}^1, \vect{\s}^2, \cdots, \vect{\s}^Q\in \cube_m$. 
Then the error is always bounded by
\begin{align}
\label{eqn::error_at_m}
\begin{aligned}
 & \norm{ \int_{\cube_m} \intterm_{m}^{M}(s_{m}, \cdots, s_1) (\rho_0) \dd \vect{s}_{1:m}  - \sum_{q=1}^{Q} w_q \intterm_{m}^{M}(\vect{\s}^q) (\rho_0)}_{1} \\
\le  & \norm{\lbop_L}^{m}_1
\sum_{
\substack{\norm{\vect{\alpha}}_{\infty} \le M-m\\ \norm{\vect{\beta}}_{\infty} \le M-m}
} 
\left(\begin{aligned}
& \frac{\norm{J}_{\infty}^{\sum_{j} \alpha_j + \beta_j}
 \dt^{\sum_{j} \alpha_j+\beta_j}}{\prod_{j=0}^{m}\alpha_j! \prod_{j=0}^{m}\beta_j!} \times \\
&\ \ \abs{\int_{\cube_m} \prod_{j=0}^{m} (\s_{j+1}-\s_j)^{\alpha_j+\beta_j}\ \dd\vect{s}_{1:m} - \sum_{q=1}^{Q} w_q \prod_{j=0}^{m} (\vect{\s}^q_{j+1} - \vect{\s}^q_{j})^{\alpha_j+\beta_j} }\end{aligned}\right).
\end{aligned}
\end{align}
\end{proposition}

\begin{proof}
By direct expansion,
\begin{align*}
 &\norm{\int_{\cube_m} \intterm_{m}^{M}(s_{m}, \cdots, s_1) (\rho_0) \dd \vect{s}_{1:m}  - \sum_{q=1}^{Q} w_q \intterm_{m}^{M}(\vect{\s}^q) (\rho_0)}_{1}\\
 =& \norm{\sum_{
\substack{\norm{\vect{\alpha}}_{\infty} \le M-m\\ \norm{\vect{\beta}}_{\infty} \le M-m}
}  \frac{\prod_{j=0}^{m} \dt^{\alpha_j+\beta_j}}{\prod_{j=0}^{m}\alpha_j! \prod_{j=0}^{m}\beta_j!} \left(\begin{aligned} & \int_{\cube_m} \prod_{j=0}^{m} (\s_{j+1}-\s_j)^{\alpha_j+\beta_j}\ \dd \vect{s}_{1:m} \\
&\ \  - \sum_{q} w_q \prod_{j=0}^{m} (\vect{\s}_{j+1}^q - \vect{\s}_{j}^q)^{\alpha_j+\beta_j}\end{aligned}\right) \text{Op}_{\vect{\alpha},\vect{\beta}}}_{1}\\
 \le & \sum_{
\substack{\norm{\vect{\alpha}}_{\infty} \le M-m\\ \norm{\vect{\beta}}_{\infty} \le M-m}
}  \left(\begin{aligned} &\frac{\prod_{j=0}^{m} \dt^{\alpha_j+\beta_j}}{\prod_{j=0}^{m}\alpha_j! \prod_{j=0}^{m}\beta_j!} \norm{\text{Op}_{\vect{\alpha},\vect{\beta}}}_{1} \times \\
&\ \ \abs\Big{\int_{\cube_m} \prod_{j=0}^{m} (\s_{j+1}-\s_j)^{\alpha_j+\beta_j}\ \dd \vect{s}_{1:m} - \sum_{q} w_q \prod_{j=0}^{m} (\vect{\s}_{j+1}^q - \vect{\s}_{j}^q)^{\alpha_j+\beta_j}} \end{aligned}\right)
\end{align*}
This proposition can then be immediately proved after applying Lemma~\ref{eqn::op_norm}.
\end{proof}

Let us denote a polynomial with $m-1$ degrees of freedom as follows:
\begin{align*}
\Phi_{\vect{\gamma}}(\vect{s}) := \prod_{j=0}^{m} (\s_{j+1}-\s_j)^{\alpha_j+\beta_j} \equiv  \prod_{j=0}^{m} (\s_{j+1}-\s_j)^{\gamma_j}. 
\end{align*}

Because $\text{Op}_{\vect{\alpha},\vect{\beta}}$ generally don't commute nor equal to each other, we need to find a quadrature scheme at the level $m$ with {\bf positive weights} such that it can approximate the following family of integrals without error:
\begin{align}
\label{eqn::criterion}
\begin{aligned}
&\sum_{q=1}^Q w_q \Phi_{\vect{\gamma}}(\vect{s}^q) = \int_{\cube_m} \prod_{j=0}^{m} (\s_{j+1}-\s_j)^{\gamma_j}\ \dd \vect{s}_{1:m},\\
& \forall \vect{\alpha},\vect{\beta} \text{ such that } \norm{\vect{\alpha}}_1 + \norm{\vect{\beta}}_1 \equiv \sum_{j=0}^{m}\alpha_j + \beta_j \equiv \sum_{j=0}^m \gamma_j \le M-m.
\end{aligned}
\end{align}

In what follows, we first derive a few particular schemes for $M = 2, 3, 4$ which are more practical than extremely high-order cases ($M\gg 1$).

\subsection{Case $M=2$}
\label{subsec::M2}
We only need to consider the case $m=1$. We need a quadrature scheme to approximate
\begin{align*}
\sum_{q=1}^Q w_q (1-\vect{s}^q_1)^{\gamma_1} (\vect{s}^q_1)^{\gamma_0} = \int_{0\le \s_1\le 1} (1-\s_1)^{\gamma_1} \s_1^{\gamma_0}\ \dd \s_1
\end{align*}
exactly without error for $(\gamma_0, \gamma_1) = (0,1), (1,0), (0,0)$.

If we only uses one data point, then the only possibility is the midpoint scheme, namely, 
\begin{align*}
w_1 = 1, \qquad \vect{\s}^1 = \begin{bmatrix} \nicefrac{1}{2} \end{bmatrix}.
\end{align*}
This leads into the midpoint scheme:
\begin{align}
\alg_{\dt}^{(\text{un},2, \text{MP})} := \krausJ{2}(\dt)(\cdot) + \dt \krausJ{1}(\dt/2) \lbop_L \krausJ{1}(\dt/2) (\cdot) + \frac{\dt^2}{2} \lbop_L^2 (\cdot).
\end{align}

 If we allow two samples, then for any $\theta \in [0,1]$, we have
\begin{align}
\label{eqn::quad_order_2_2}
\begin{aligned}
&w_1 \in [0,\nicefrac{1}{2}],  \qquad &\vect{\s}^1 = \begin{bmatrix} \theta \end{bmatrix};\\
&w_2 = 1-w_1,\qquad  &\vect{\s}^2 = \begin{bmatrix} \frac{\nicefrac{1}{2} - w_1 \theta}{1-w_1} \end{bmatrix}.
\end{aligned}
\end{align}

When $w_1 = w_2 = 1/2$ and $\theta = 1/2$, it reduces to the above midpoint rule. Another choice is that $\theta = 0$ and $w_1 = w_2 = 1/2$, and it leads into the Trapezoidal rule with the following scheme:
\begin{align}
\alg_{\dt}^{(\text{un},2,\text{TP})} := \krausJ{2}(\dt)(\cdot) + \frac{\dt}{2}\krausJ{1}(\dt) \lbop_L  + \frac{\dt}{2}  \lbop_L \krausJ{1}(\dt) + \frac{\dt^2}{2} \lbop_L^2.
\end{align}

\subsection{Case $M=3$}
\label{subsec::M3}

We need to consider two cases: $m=1$ and $m=2$:

\begin{itemize}[leftmargin=2em]
\item When $m=1$, we need a quadrature scheme with precise evaluations of 
\begin{align*}
\int_{0\le \s_1 \le 1} (1-\s_1)^{\gamma_1} \s_1^{\gamma_0}\ \dd \s_1 = \frac{\gamma_1!\gamma_0!}{(\gamma_0 + \gamma_1+1)!}
\end{align*}
for $(\gamma_0, \gamma_1) = (0,0), (0,1), (0,2), (1,0), (1,1), (2,0)$.
With direct calculations, it turns out that we only need two samples:
\begin{align*}
w_1 = \frac{3 (1-2 \theta)^2}{4(1-3 \theta+3\theta^2)},  & \qquad \vect{\s}^1 = \begin{bmatrix}\frac{-2 + 3 \theta}{-3 + 6\theta}\end{bmatrix}; \\
w_2 = \frac{1}{4(1-3\theta +3 \theta^2)}, &\qquad  \vect{\s}^2 = \begin{bmatrix}\theta\end{bmatrix}.
\end{align*}
To ensure that $\vect{\s}^1$ fall into the correct range $[0,1]$, we need
\begin{align*}
\theta \in [0, \nicefrac{1}{3}]\cup [\nicefrac{2}{3}, 1].
\end{align*}

Similar to the Trapezoidal scheme above, if $\vect{\s}^q$ has elements either $0$ or $1$, then we can uses one less Kraus operator, and hence, we end up with the following two possible choices:
\begin{align}
w_1 = \nicefrac{3}{4}, \qquad w_2 =\nicefrac{1}{4}, \qquad \vect{\s}^1 = \begin{bmatrix}\nicefrac{2}{3}\end{bmatrix}, \qquad \vect{\s}^2 = \begin{bmatrix} 0\end{bmatrix}. \label{eqn::quad_third_1}\\
w_1 = \nicefrac{3}{4}, \qquad w_2 =\nicefrac{1}{4}, \qquad \vect{\s}^1 = \begin{bmatrix}\nicefrac{1}{3}\end{bmatrix}, \qquad \vect{\s}^2 = \begin{bmatrix}1\end{bmatrix}. \label{eqn::quad_third_2}
\end{align}

\item When $m=2$, we need a quadrature scheme with exact approximations of 
\begin{align*}
 \int_{0\le \s_1 \le \s_2 \le 1} (1-\s_2)^{\gamma_2} (\s_2 - \s_1)^{\gamma_1} (\s_1)^{\gamma_0}\ \dd \vect{s}_{1:2}
\end{align*}
for $(\gamma_0,\gamma_1,\gamma_2) = (1,0,0),(0,1,0),(0,0,1),(0,0,0)$. In this case, we only need one sample points:
\begin{align*}
w_1 = \nicefrac{1}{2} \qquad \vect{\s}^1 = \begin{bmatrix} \nicefrac{1}{3} & \nicefrac{2}{3}\end{bmatrix},
\end{align*}
\end{itemize}

We obtain the following two third-order schemes:
\begin{align}
\label{eqn::third_order::1}
\begin{aligned}
\alg_{\dt}^{(\text{un},3)}(\rho) =& \krausJ{3}(\dt)(\rho) \\
& + \frac{3\dt}{4} \krausJ{2}(\nicefrac{\dt}{3}) \lbop_L \krausJ{2}(\nicefrac{2\dt}{3}) (\rho) \\
& + \frac{\dt}{4}\krausJ{2}(\dt) \lbop_L (\rho) \\
&+ \frac{\dt^2}{2} \krausJ{1} (\nicefrac{\dt}{3})\lbop_L \krausJ{1}(\nicefrac{\dt}{3}) \lbop_L \krausJ{1}(\nicefrac{\dt}{3}) (\rho) \\
&+ \frac{\dt^3}{6}\lbop_L^3 (\rho),
\end{aligned}
\end{align}
and
\begin{align}
\label{eqn::third_order::2}
\begin{aligned}
\alg_{\dt}^{(\text{un},3)}(\rho) =& \krausJ{3}(\dt)(\rho) \\
&+ \frac{3\dt}{4} \krausJ{2}(\nicefrac{2\dt}{3}) \lbop_L \krausJ{2}(\nicefrac{\dt}{3}) (\rho) \\
&+ \frac{\dt}{4}  \lbop_L \krausJ{2}(\dt) (\rho) \\
&+ \frac{\dt^2}{2} \krausJ{1} (\nicefrac{\dt}{3})\lbop_L \krausJ{1}(\nicefrac{\dt}{3}) \lbop_L \krausJ{1}(\nicefrac{\dt}{3}) (\rho) \\
&+ \frac{\dt^3}{6}\lbop_L^3 (\rho).
\end{aligned}
\end{align}

\subsection{Case $M=4$}
\label{subsec::M4}

We similarly need to handle cases $m=1, 2, 3$ separately:
\begin{itemize}[leftmargin=2em]
\item When $m = 1$, we need to ensure that 
\begin{align*}
\int_{0\le \s_1\le 1} (1-\s_1)^{\gamma_1} \s_1^{\gamma_0}\ \dd \s_1 = \frac{\gamma_1!\gamma_0!}{(1+\gamma_0+\gamma_1)!},
\end{align*}
for $(\gamma_0, \gamma_1)$ such that $\gamma_0 + \gamma_1 \le 3$ are precisely approximated via samples.
By solving the conditions \eqref{eqn::criterion}, we notice that only two samples are necessary, and these are
\begin{align*}
w_1 = \nicefrac{1}{2}, \qquad \vect{\s}^1  = \begin{bmatrix} \frac{3+\sqrt{3}}{6} \end{bmatrix},\\
w_2 = \nicefrac{1}{2}, \qquad \vect{\s}^2 = \begin{bmatrix}\frac{3-\sqrt{3}}{6}\end{bmatrix}.
\end{align*}
Moreover, this solution is unique for the choice of two samples.

\item When $m=2$, we need to ensure that
\begin{align*}
\int_{0\le \s_1 \le \s_2\le 1} (1-\s_2)^{\gamma_2} (\s_2-\s_1)^{\gamma_1} \s_1^{\gamma_0} \dd \s_1\dd\s_2 = \frac{\gamma_0!\gamma_1!\gamma_2!}{(\gamma_0 + \gamma_1 + \gamma_2+2)!}
\end{align*}
can be exactly computed for any $(\gamma_0, \gamma_1, \gamma_2)$ with $\gamma_0+\gamma_1+\gamma_2 \le 2$.
By Mathematica, there is no solution using only two samples. For three samples, there are, in fact, an infinite amount of possible solutions. The generic forms are 
\begin{align*}
w_1 = \frac{y (6 z-4)-4 z+3}{12 (x-y) (x-z)}, w_2 = \frac{x (4-6 z)+4 z-3}{12 (x-y) (y-z)}, w_3 = \frac{x (6 y-4)-4 y+3}{12 (x-z)
   (y-z)},\\
 \vect{s}^1 = \begin{bmatrix}\frac{1}{4} \left(2 x-\frac{\sqrt{2} \sqrt{-((x (6 y-4)-4 y+3) (x (6 z-4)-4 z+3) (y (6 z-4)-4 z+3))}}{y (6 z-4)-4
   z+3}\right) & x\end{bmatrix};\\
 \vect{s}^2 = \begin{bmatrix}\frac{1}{4} \left(2 y-\frac{\sqrt{2} \sqrt{-((x (6 y-4)-4 y+3) (x (6 z-4)-4 z+3) (y (6 z-4)-4 z+3))}}{x (6
   z-4)-4 z+3}\right) & y\end{bmatrix};\\
 \vect{s}^3 = \begin{bmatrix}\frac{1}{4} \left(2 z-\frac{\sqrt{2} \sqrt{-((x (6 y-4)-4 y+3) (x (6 z-4)-4 z+3) (y (6 z-4)-4
   z+3))}}{x (6 y-4)-4 y+3}\right) & z\end{bmatrix}.
\end{align*}
or 
\begin{align*}
w_1 = \frac{y (6 z-4)-4 z+3}{12 (x-y) (x-z)}, w_2 = \frac{x (4-6 z)+4 z-3}{12 (x-y) (y-z)},w_3 = \frac{x (6 y-4)-4 y+3}{12 (x-z)
   (y-z)},\\
 \vect{s}^1 =\begin{bmatrix} \frac{1}{4} \left(2x + \frac{\sqrt{2} \sqrt{-((x (6 y-4)-4 y+3) (x (6 z-4)-4 z+3) (y (6 z-4)-4 z+3))}}{y (6 z-4)-4
   z+3}\right) & x \end{bmatrix};\\
 \vect{s}^2 = \begin{bmatrix} \frac{1}{4} \left(2y + \frac{\sqrt{2} \sqrt{-((x (6 y-4)-4 y+3) (x (6 z-4)-4 z+3) (y (6 z-4)-4 z+3))}}{x (6
   z-4)-4 z+3}\right) & y \end{bmatrix};\\
 \vect{s}^3 = \begin{bmatrix}\frac{1}{4} \left(2z + \frac{\sqrt{2} \sqrt{-((x (6 y-4)-4 y+3) (x (6 z-4)-4 z+3) (y (6 z-4)-4
   z+3))}}{x (6 y-4)-4 y+3}\right) & z\end{bmatrix}.
\end{align*}
Not all values of $(x,y,z)\in [0,1]^3$ provides a valid scheme.
When $x=1/4$, $y=3/4$, $z=1$, the first generic solution provides the following explicit possible samples:
\begin{align}
\begin{aligned}
w_1 = \nicefrac{1}{9}, \qquad &\vect{s}^1 = \begin{bmatrix} 0 & \nicefrac{1}{4} \end{bmatrix};\\
w_2 = \nicefrac{1}{3},\qquad &\vect{s}^2 = \begin{bmatrix} \nicefrac{1}{2} & \nicefrac{3}{4} \end{bmatrix};\\ 
w_3= \nicefrac{1}{18},\qquad & \vect{s}^3 = \begin{bmatrix} 0 & 1 \end{bmatrix}.
\end{aligned}
\end{align}

\item When $m=3$, we need to approximate
\begin{align*}
\int_{0\le \s_1 \le \s_2 \le \s_3\le 1} (1-\s_3)^{\gamma_3} (1-\s_2)^{\gamma_2} (\s_2 - \s_1)^{\gamma_1} (\s_1)^{\gamma_0}\ \dd \vect{\s}
\end{align*}
for $(\gamma_0,\gamma_1,\gamma_2, \gamma_3) = (1,0,0,0),(0,1,0,0),(0,0,1,0),(0,0,0,1),(0,0,0,0)$. In this case, we only need one sample:
\begin{align*}
w_1 = \nicefrac{1}{6} \qquad \vect{\s}^1 =\begin{bmatrix} \nicefrac{1}{4} &  \nicefrac{2}{4} &  \nicefrac{3}{4} \end{bmatrix}.
\end{align*}
\end{itemize}

Therefore, a possible $4^{\text{th}}$ order scheme is:
\begin{align}
\label{eqn::fourth_order}
\begin{aligned}
\alg_{\dt}^{(\text{un},4)}(\rho) =& \krausJ{4}(\dt)(\rho) \\
&+ \frac{\dt}{2}\krausJ{3}(\frac{3-\sqrt{3}}{6}\dt)\lbop_L  \krausJ{3}(\frac{3+\sqrt{3}}{6}\dt) (\rho) \\
&+ \frac{\dt}{2}\krausJ{3}(\frac{3+\sqrt{3}}{6}\dt)\lbop_L  \krausJ{3}(\frac{3-\sqrt{3}}{6}\dt) (\rho)\\
& + \frac{\dt^2}{9} \krausJ{2}(\frac{3}{4}\dt)\lbop_L \krausJ{2}(\frac{1}{4}\dt)\lbop_L (\rho) \\
&+ \frac{\dt^2}{3} \krausJ{2}(\frac{1}{4}\dt)\lbop_L \krausJ{2}(\frac{1}{4}\dt)\lbop_L\krausJ{2}(\frac{1}{2}\dt) (\rho) \\
&  + \frac{\dt^2}{18} \lbop_L \krausJ{2}(\dt) \lbop_L (\rho)\\
&+ \frac{\dt^3}{6}\krausJ{1}(\nicefrac{\dt}{4})\lbop_L \krausJ{1}(\nicefrac{\dt}{4})\lbop_L \krausJ{1}(\nicefrac{\dt}{4})\lbop_L \krausJ{1}(\nicefrac{\dt}{4}) (\rho) \\
&+\frac{\dt^4}{24} \lbop_L^4 (\rho).
\end{aligned}
\end{align}

\subsection{Some general properties of the selections of quadrature schemes}

We summarize here a few general properties arising from the selection of quadrature schemes at different levels $m$. 
\begin{lemma}
Given an $\eta\in \Natural$, to make sure 
\begin{align}
\label{eqn::criterion_1}
\begin{aligned}
&\sum_{q=1}^Q w_q (1-\vect{s}_1^q)^{\gamma_1} (\vect{s}_1^q)^{\gamma_0} = \int_{0\le \s_1  \le 1} (1-s_1)^{\gamma_1} s_1^{\gamma_0}\ \dd s_{1},\\ 
& \forall \gamma_0, \gamma_1 \text{ with } \gamma_0 + \gamma_1 \le \eta,
\end{aligned}
\end{align}
it is equivalent to ensure that for $\phi_{\gamma}(x) = x^{\gamma}$, 
\begin{align}
\label{eqn::criterion_1_equiv}
\sum_{q=1}^Q w_q \phi_{\gamma}(\vect{s}_1^q) = \int_{0\le \s_1  \le 1} \phi_{\gamma}(\vect{s}_1^q)\ \dd s_{1},\qquad \forall \gamma \text{ with } \gamma \le \eta.
\end{align}
Namely, it is a scheme with positive weights on $[0,1]$ such that it has algebraic precision at least $\eta$.
\end{lemma}

\begin{proof}
If \eqref{eqn::criterion_1} holds, then clearly \eqref{eqn::criterion_1_equiv} holds. Backwardly, if \eqref{eqn::criterion_1_equiv} holds, then for any given $\gamma_0, \gamma_1$, $(1 - s)^{\gamma_1} s^{\gamma_0} = p(s)$ is a polynomial of $s$ with degree $\gamma_0+\gamma_1\le \eta$, and thus
\begin{align*}
\int_{0\le \s_1  \le 1} (1-s_1)^{\gamma_1} s_1^{\gamma_0}\ \dd s_{1} = \int_{0\le s_1 \le 1} p(s_1)\ \dd s_1\ \myeq{\eqref{eqn::criterion_1_equiv}}\  \sum_{q=1}^Q w_q\ p(\vect{s}_1^q) = \sum_{q=1}^Q w_q (1-\vect{s}_1^q)^{\gamma_1} (\vect{s}_1^q)^{\gamma_0}.
\end{align*}
\end{proof}

\begin{lemma}[Case $m=1$]
When $m = 1$, to satisfies the criterion \eqref{eqn::criterion}, we at most need $\ceil{M/2}$ samples and these samples points can be chosen as Gaussian nodes (for the interval $[0,1]$).
\end{lemma}

\begin{proof}
By the above lemma, we just need to find positive weights $w_q$ and samples $\vect{s}_1^q$ such that the quadrature reaches  at least $M-1$ algebraic precision. This can be ensured by Gaussian quadrature which uses $k$ points with guaranteed algebraic precision $2k-1$ and all weights are guaranteed to be positive. Then the conclusion easily follows.
\end{proof}

\begin{lemma}[Case $m = M-1$]
When $m = M-1$, to satisfies the criterion \eqref{eqn::criterion}, we only need one sample with 
\begin{align*}
w_1 = \frac{1}{m!} = \frac{1}{(M-1)!}, \qquad \vect{s}^1 = \underbrace{\begin{bmatrix}\frac{1}{M}, \frac{2}{M}, \cdots, \frac{M-1}{M}\end{bmatrix}}_{\text{has } M-1 \text{ elements}}.
\end{align*}
\end{lemma}

\begin{proof}
The proof is trivial from the definitions and is thus omitted.
\end{proof}

\section{Proof of \thmref{thm::error_anal}: Details of the error analysis}
\label{app::err_details}

\subsection{Error from truncation in the series expansion}

\begin{proposition}
For any $\dt > 0$, 
\begin{align}
\label{eqn::step_1}
\begin{aligned}
\ttnormBig{\rho_{\dt} - e^{\lbop_J \dt} (\rho_0)  -  &\sum_{m=1}^{M} \int\limits_{0\le s_1 \le \cdots \le s_m \le \dt}
e^{\lbop_J (\dt - s_{m})} \lbop_L e^{\lbop_J (s_{m} - s_{m-1})} \lbop_L \cdots
\\ &\hspace{27ex}
e^{\lbop_J (s_2 - s_1)} \lbop_L e^{\lbop_J s_1} (\rho_0)\ \dd \vect{s}_{1:m}
} \\
&\le  \frac{1}{(M+1)!} \ttnorm{\lbop_L}^{M+1} \dt^{M+1}.
\end{aligned}
\end{align}
\end{proposition}

\begin{proof}
We need to quantify the remainder term in \eqref{eqn::series_expansion}. Recall that for any $t\ge 0$ and any positive semi-definite matrix $\rho$, we have $\ttnorm{e^{\lbop_J t}(\rho)} \le \ttnorm{\rho}$ in \eqref{eqn::op_lbopJ_norm}. Hence, we know that
\begin{align*}
& \ttnormBig{\int\limits_{0\le s_1 \le \cdots \le s_{M+1} \le \dt} e^{\lbop_J (\dt - s_{M+1})}\lbop_L e^{\lbop_J (s_{M+1} - s_{M})}\lbop_L \cdots e^{\lbop_J (s_2- s_1)}\lbop_L (\rho_{s_1})\ \dd \vect{s}_{1:{M+1}}} \\
& \hspace{2em} \le   \int\limits_{0\le s_1 \le \cdots \le s_{M+1} \le \dt}  \ttnorm{\lbop_L}^{M+1} \ \dd s_1 \cdots \dd s_{M+1}  \le  \frac{1}{(M+1)!} \ttnorm{\lbop_L}^{M+1} \dt^{M+1}.
\end{align*}
Therefore, we immediately have \eqref{eqn::step_1}.
\end{proof}

\subsection{Error from approximating $e^{\lbop_J (t-s)}$}
Recall that in the Step (\rom{2}), we approximate $e^{\lbop_J(t-s)}$ by $\krausJ{\alpha}(t-s)$. Therefore, We shall first quantify the difference
$\ttnorm{e^{\lbop_J (t-s)} - \krausJ{\alpha}(t-s)},$ for any integer $\alpha \ge 0$.

\begin{lemma}
\label{lem::lbopJ_diff}
If $0 \le t\le \frac{1}{\norm{J}_{\infty}}$, we have 
\begin{align}
\label{eqn::lbopJ_diff}
     \ttnormbig{e^{\lbop_J t} - \krausJ{\alpha}(t)}\le   \frac{3e^{2\norm{J}_{\infty} t}}{(\alpha+1)!} \norm{J}_{\infty}^{\alpha+1} t^{\alpha+1}.
\end{align}
\item For an arbitrary $t\ge 0$,
\begin{align}
\label{eqn::K_alpha_bound}
\ttnormbig{\krausJ{\alpha}(t)} \le e^{2 \norm{J}_{\infty} t}.
\end{align}
\end{lemma}

\begin{proof}
Recall from \lemref{lem::evol_exp_lbopJ} that $e^{\lbop_J t}(\rho) = \KRopbig{e^{J t}}(\rho)$.
Let us decompose $e^{J t}$ as
\begin{align*}
e^{J t} =  \underbrace{\sum_{k=0}^{\alpha} \frac{J^k t^k}{k!}}_{=: J_\alpha} + \underbrace{\sum_{k=\alpha+1}^{\infty}  \frac{J^k t^k}{k!}}_{=:J_{\text{Rem}}}.
\end{align*}
It is not hard to estimate that
\begin{align*}
&\norm{J_{\alpha}}_{\infty} \le e^{\norm{J}_{\infty} t},
\end{align*}
and that 
\begin{align*}
\norm{J_{\text{Rem}}}_{\infty} &\le \big(\norm{J}_{\infty} t\big)^{\alpha+1}\sum_{k'=0}^{\infty} \frac{\big(\norm{J}_{\infty} t \big)^{k'}}{(\alpha+1+k')!}  \\
& \le \big(\norm{J}_{\infty} t\big)^{\alpha+1}\sum_{k'=0}^{\infty} \frac{\big(\norm{J}_{\infty} t\big)^{k'}}{(\alpha+1)!k'!}  
 \le \frac{e^{\norm{J}_{\infty} t}}{(\alpha+1)!} \norm{J}_{\infty}^{\alpha+1} t^{\alpha+1}.
\end{align*}
Recall from \eqref{eq::approxexpJ} that $\krausJ{\alpha}(t) = \KRopbig{J_{\alpha}}$. Therefore, for any $\rho$ with $\ttnorm{\rho}=1$, 
\begin{align*}
\begin{aligned}
\ttnormbig{e^{\lbop_J t }(\rho) - \krausJ{\alpha}(t)(\rho)} &=  \ttnormbig{J_{\alpha} \rho J_{\text{Rem}}^{\dagger} + J_{\text{Rem}} \rho J_{\alpha}^{\dagger} + J_{\text{Rem}} \rho J_{\text{Rem}}^{\dagger}}\\
& \le  2 \norm{J_{\alpha}}_{\infty} \norm{J_{\text{Rem}}}_{\infty} + \norm{J_{\text{Rem}}}_{\infty}^2 \\
& \le \frac{e^{2\norm{J}_{\infty}t}}{(\alpha+1)!} \norm{J}_{\infty}^{\alpha+1} t^{\alpha+1} \big(2  + \frac{(\norm{J}_{\infty} t)^{\alpha+1}}{(\alpha+1)!})\\
&\le \frac{3e^{2\norm{J}_{\infty} t}}{(\alpha+1)!} \norm{J}_{\infty}^{\alpha+1} t^{\alpha+1} \qquad \text{ (by } t \norm{J}_{\infty} \le 1 \text{)},
\end{aligned}
\end{align*}
which gives \eqref{eqn::lbopJ_diff}.
As for the norm $\ttnorm{\krausJ{\alpha}(t)}$, it follows immediately from  \eqref{eq::approxexpJ} that $\krausJ{\alpha}(t) = \KRopbig{J_{\alpha}}$ and thus
$\ttnormbig{\krausJ{\alpha}(t)} \le \norm{J_{\alpha}}_{\infty}^2 \le e^{2\norm{J}_{\infty} t}.$
\end{proof}

Next, we can quantify the error bound for Steps (\rom{1}) and (\rom{2}), summarized in the following proposition.

\begin{proposition}
\label{prop::step_one_two_bound}
When $\dt \le \frac{1}{\norm{J}_{\infty}}$, for a general order $M$, we have
\begin{align*}
\begin{aligned}
& \begin{aligned}
\ttnormbig{\rho_{\dt} - \krausJ{M}(\dt) (\rho_0) - \sum_{m=1}^{M} \int\limits_{0\le s_1 \le \cdots \le s_m \le \dt} &\krausJ{M-m}(\dt- s_{m}) \lbop_L
\krausJ{M-m} (s_{m}-s_{m-1}) \lbop_L \cdots \\
&\hspace{1em} \krausJ{M-m}(s_2-s_1)\lbop_L \krausJ{M-m}(s_1) (\rho_0)\ \dd \vect{s}_{1:m}} \\
\end{aligned}\\
&\hspace{6em} 
\le \frac{23}{(M+1)!} \big(\norm{J}_{\infty} + \norm{\lbop_L}_{1}\big)^{M+1} \dt^{M+1}.\end{aligned}
\end{align*}
\end{proposition}

\begin{proof}
For the zeroth-order term, by \lemref{lem::lbopJ_diff}, we get
\begin{align}
\label{eqn::LbopJ_zero_order}
\ttnormbig{e^{\lbop_J \dt}(\rho_0) - \krausJ{M}(\dt) (\rho_0)}
 \le  \frac{3 e^{2 \norm{J}_{\infty} \dt}}{(M +1)!} \norm{J}_{\infty}^{M+1} \dt^{M +1}.
\end{align}

As for errors from the approximation for the $m^{\text{th}}$-order term in \eqref{eqn::step_two_approx} (with $1\le m\le M$),
\begin{align*}
& \begin{aligned}
& \ttnormbig{e^{\lbop_J (\dt-s_m)} \lbop_L e^{\lbop_J (s_m-s_{m-1})} \lbop_L \cdots e^{\lbop_J (s_2-s_1)} \lbop_L e^{\lbop_J (s_1 - 0)} (\rho_0) \\
&\hspace{1em} - \krausJ{M-m}(\dt - s_{m}) \lbop_L \krausJ{M-m} (s_{m} - s_{m-1}) \lbop_L \cdots \krausJ{M-m}(s_2 - s_1)\lbop_L \krausJ{M-m}(s_1) (\rho_0)} \\
\end{aligned} \\
& = \ttnormBig{\sum_{\beta=0}^{m} \krausJ{M-m}(\dt - s_m) \lbop_L \cdots \bigl(e^{\lbop_J(s_{\beta+1} - s_{\beta})} - \krausJ{M-m}(s_{\beta+1} - s_{\beta}) \bigr)\lbop_L \cdots e^{\lbop_J(s_1-0)}(\rho_0)} \\
& \myle{\eqref{eqn::op_lbopJ_norm},\eqref{eqn::lbopJ_diff},\eqref{eqn::K_alpha_bound}}\ \ \ \ \sum_{\beta=0}^{m} \left(\begin{aligned} &  \ttnorm{\lbop_L}^{m} e^{2\norm{J}_{\infty} (\dt - s_m)} e^{2\norm{J}_{\infty} (s_{m} - s_{m-1})} \cdots \\
&\qquad 3 e^{2 \norm{J}_{\infty} (s_{\beta+1}-s_{\beta})} \frac{(s_{\beta+1}-s_{\beta})^{M-m+1} \norm{J}_{\infty}^{M-m+1}}{(M-m +1)!}\end{aligned}\right)\\
&= \frac{3}{(M-m +1)!} \ttnorm{\lbop_L}^{m} \norm{J}_{\infty}^{M-m+1} \sum_{\beta=0}^{m} e^{2\norm{J}_{\infty} (\dt - s_\beta)} (s_{\beta+1}-s_{\beta})^{M-m+1}\\
&\le \frac{3}{(M-m +1)!} \ttnorm{\lbop_L}^{m} \norm{J}_{\infty}^{M-m+1} e^{2\norm{J}_{\infty} \dt}\sum_{\beta=0}^{m}  (s_{\beta+1}-s_{\beta})^{M-m+1}\\
&= \frac{3}{(M-m +1)!} \ttnorm{\lbop_L}^{m} \norm{J}_{\infty}^{M-m+1} e^{2\norm{J}_{\infty} \dt} \dt^{M-m+1} \sum_{\beta=0}^{m} \Big(\frac{s_{\beta+1}-s_{\beta}}{\dt}\Big)^{M-m+1}\\
&\le \frac{3}{(M-m+1)!} \ttnorm{\lbop_L}^{m} \norm{J}_{\infty}^{M-m+1} e^{2\norm{J}_{\infty}\dt}\dt^{M-m+1} \sum_{\beta=0}^{m} \frac{s_{\beta+1}-s_{\beta}}{\dt}\\
&=\frac{3}{(M-m+1)!} \ttnorm{\lbop_L}^{m} \norm{J}_{\infty}^{M-m+1} e^{2\norm{J}_{\infty}\dt}\dt^{M-m+1}.
\end{align*}
In the above, we have denoted $s_{m+1}\equiv \Delta t$ and $s_0 = 0$ for convenience, and in the second last line, we used the fact that $\sum_{\beta} s_{\beta+1}-s_{\beta} = \dt$.
Therefore,
\begin{align}
\label{eqn::err_approx_lbopJ}
\begin{aligned}
\ttnormBig{ & \int\limits_{0\le s_1 \le \cdots \le s_m \le \dt} e^{\lbop_J (\dt - s_{m})} \lbop_L e^{\lbop_J (s_{m} - s_{m-1})} \lbop_L \cdots e^{\lbop_J (s_2 - s_1)} \lbop_L e^{\lbop_J s_1} (\rho_0)\ \vect{s}_{1:m} \\
&- \int\limits_{0\le s_1 \le \cdots \le s_m \le \dt} \krausJ{M-m}(\dt - s_{m}) \lbop_L \cdots
\krausJ{M-m}(s_2 - s_1)\lbop_L \krausJ{M-m}(s_1) (\rho_0)\ \dd \vect{s}_{1:m}
} \\
&\le \frac{3}{m! (M-m+1)!} \ttnorm{\lbop_L}^{m}  \norm{J}_{\infty}^{M-m+1} e^{2\norm{J}_{\infty}\dt}\dt^{M+1}.
\end{aligned}
\end{align}

By combining the last equation with the estimate in \eqref{eqn::step_1}, we have
\begin{align*}
& \begin{aligned}
\ttnormBig{\rho_{\dt} - \krausJ{M}(\dt) (\rho_0) - \sum_{m=1}^{M} \int\limits_{0\le s_1 \le \cdots \le s_m \le \dt} &\krausJ{M-m}(\dt - s_{m}) \lbop_L
\krausJ{M-m} (s_{m} - s_{m-1}) \lbop_L \cdots \\
& \krausJ{M-m}(s_2 - s_1)\lbop_L \krausJ{M-m}(s_1) (\rho_0)\ \dd \vect{s}_{1:m}} \\
\end{aligned}\\
& \le \underbrace{\frac{1}{(M+1)!} \ttnorm{\lbop_L}^{M+1} \dt^{M+1}}_{\text{by } \eqref{eqn::step_1}}
+ \underbrace{3 e^{2 \norm{J}_{\infty} \dt} \frac{\norm{J}_{\infty}^{M+1}\dt^{M+1}}{(M+1)!}}_{\text{by }\eqref{eqn::LbopJ_zero_order}}\\
&\qquad + \underbrace{\sum_{m=1}^{M} \frac{3}{m! (M-m+1)!} \ttnorm{\lbop_L}^{m} \norm{J}_{\infty}^{M-m+1} e^{2\norm{J}_{\infty}\dt}\dt^{M+1}}_{\text{by } \eqref{eqn::err_approx_lbopJ}} \\
&= \Big(\frac{1}{(M+1)!} \ttnorm{\lbop_L}^{M+1}+ \sum_{m=0}^{M} \frac{3}{m! (M-m+1)!} \ttnorm{\lbop_L}^{m} \norm{J}_{\infty}^{M-m+1} e^{2\norm{J}_{\infty}\dt}\Big)\dt^{M+1}\\
&\le \Big(\sum_{m=0}^{M+1} \frac{3}{m! (M-m+1)!} \ttnorm{\lbop_L}^{m} \norm{J}_{\infty}^{M-m+1} e^{2\norm{J}_{\infty}\dt}\Big)\dt^{M+1}\\
&= \frac{3 e^{2\norm{J}_{\infty}\dt}}{(M+1)!}\Big(\sum_{m=0}^{M+1} \frac{(M+1)!}{m! (M-m+1)!} \ttnorm{\lbop_L}^{m} \norm{J}_{\infty}^{M-m+1} \Big)\dt^{M+1}\\
&\le \frac{23}{(M+1)!} \big(\norm{J}_{\infty} + \norm{\lbop_L}_{1}\big)^{M+1} \dt^{M+1} \qquad (\text{by } \norm{J}_{\infty} \dt \le 1).
\end{align*}
\end{proof}

\subsection{Error bound for quadrature schemes}

As observed above, the case $M=1$ does not require further quadrature approximations. We summarize the quadrature approximation errors for a general order $M \ge 2$ below.

The generic bound in \propref{prop::error_1} can be simplified by choosing an appropriate quadrature scheme that satisfies \eqref{eqn::criterion}:
\begin{proposition}[A simplified error bound]
\label{prop::quad_error_bound}
Suppose a quadrature scheme with positive weights $\vect{w}=(w_1, w_2, \cdots, w_Q)$ and samples $\vect{\s}^1, \vect{\s}^2, \cdots, \vect{\s}^Q\in \cube_m$ satisfy the condition \eqref{eqn::criterion} and assume that $\dt\norm{J}_{\infty} \le 1$, then the quadrature approximation error is 
\begin{align}
\begin{aligned}
& \norm{ \int_{\cube_m} \intterm_{m}^{M}(s_{m}, \cdots, s_1) (\rho_0) \dd \vect{s}_{1:m}  - \sum_{q=1}^{Q} w_q \intterm_{m}^{M}(\vect{\s}^q) (\rho_0)}_{1} \\
 & \le  \frac{2 (M-m)!}{M!} \norm{\lbop_L}_1^{m} \dt^{M-m+1} \norm{J}_{\infty}^{M-m+1} C(M,m),
 \end{aligned}
\end{align}
where
\begin{align}
\label{eqn::CMm}
C(M, m) := 
\sum_{
\substack{\norm{\vect{\alpha}}_{\infty} \le M-m\\ \norm{\vect{\beta}}_{\infty} \le M-m\\ \sum_{j=0}^m \alpha_j + \beta_j \ge M-m+1}} 
\frac{1}{\prod_{j=0}^{m}\alpha_j! \prod_{j=0}^{m}\beta_j!}.
\end{align}
\end{proposition}

\begin{proof}
With the condition \eqref{eqn::criterion}, the above \propref{prop::error_1} gives
\begin{align*}
 & \norm{ \int_{\cube_m} \intterm_{m}^{M}(s_{m}, \cdots, s_1) (\rho_0) \dd \vect{s}_{1:m}  - \sum_{q=1}^{Q} w_q \intterm_{m}^{M}(\vect{\s}^q) (\rho_0)}_{1} \\
\le  & \norm{\lbop_L}^{m}_1
\sum_{
\substack{\norm{\vect{\alpha}}_{\infty} \le M-m\\ \norm{\vect{\beta}}_{\infty} \le M-m\\ \sum_{j=0}^m \alpha_j + \beta_j \ge M-m+1}
} 
\left(\begin{aligned} & \frac{\norm{J}_\infty^{\sum_{j} \alpha_j + \beta_j}
 \dt^{\sum_{j} \alpha_j+\beta_j}}{\prod_{j=0}^{m}\alpha_j! \prod_{j=0}^{m}\beta_j!} \times \\
& \abs{\int_{\cube_m} \Phi_{\vect{\gamma}}(\vect{s})\ \dd\vect{s}_{1:m} - \sum_{q=1}^{Q} w_q \prod_{j=0}^{m} (\vect{\s}^q_{j+1} - \vect{\s}^q_{j})^{\alpha_j+\beta_j} }
\end{aligned}\right).
\end{align*}

It is a well-known integration result in hyper-cube that
\begin{align*}
\int_{\cube_m} \Phi_{\vect{\gamma}}(\vect{s})\ \dd\vect{s}_{1:m} \equiv & \int_{\cube_m} \prod_{j=0}^{m} (\s_{j+1}-\s_j)^{\alpha_j+\beta_j}\ \dd\vect{s}_{1:m} 
= \frac{\gamma_0!\gamma_1!\cdots \gamma_m!}{(\gamma_0+\gamma_1+\cdots+\gamma_m+m)!}.
\end{align*}
If we reduce one $\gamma_j$, the expression is always non-decreasing. This leads into the following under  the assumption that $\sum_{j=0}^m \gamma_j \ge M-m+1$:
\begin{align}
\label{eqn::int_bound}
\int_{\cube_m} \Phi_{\vect{\gamma}}(\vect{s})\ \dd\vect{s}_{1:m}  \le& \max_{\sum_{j=0}^m \gamma_j = M-m+1} \frac{\gamma_0!\gamma_1!\cdots \gamma_m!}{(\gamma_0+\gamma_1+\cdots+\gamma_m+m)!} \le \frac{(M-m+1)!}{(M+1)!}.
\end{align}
This second inequality holds because for any $x, m\in \Natural$, we always have $\max_{0\le x \le m} x! (m-x)! \le m!$.

Moreover, as $\vect{s}_{j+1}^q - \vect{s}_j^q \le 1$, when $\sum_{j=0}^m \gamma_j \ge M-m+1$, 
\begin{align*}
\sum_{q=1}^{Q} w_q \prod_{j=0}^{m} (\vect{\s}^q_{j+1} - \vect{\s}^q_{j})^{\alpha_j+\beta_j}  &\le \max_{\sum_{j=0}^m \gamma'_j=M-m}  \sum_{q=1}^{Q} w_q \prod_{j=0}^{m} (\vect{\s}^q_{j+1} - \vect{\s}^q_{j})^{\gamma'_j} \\
&\myeq{\eqref{eqn::criterion}} \max_{\sum_{j=0}^m \gamma'_j=M-m} \int_{\cube_m} \Phi_{(\gamma_0', \gamma_1', \cdots, \gamma_m')}(\vect{s})\ \dd \vect{s}_{1:m}\\
&= \max_{\sum_{j=0}^m \gamma'_j=M-m} \frac{\gamma_0'! \cdots \gamma_m'!}{(m+\gamma_0'+\gamma_1'+\cdots+\gamma_m')!}\le  \frac{(M-m)!}{M!}.
\end{align*}
The second line comes from the assumption of the quadrature scheme and the third line follows the same calculations as the bound in \eqref{eqn::int_bound} above.
Hence,
\begin{align*}
 & \norm{ \int_{\cube_m} \intterm_{m}^{M}(s_{m}, \cdots, s_1) (\rho_0) \dd \vect{s}_{1:m}  - \sum_{q=1}^{Q} w_q \intterm_{m}^{M}(\vect{\s}^q) (\rho_0)}_{1} \\
\le  & \frac{2 (M-m)!}{M!} \norm{\lbop_L}^{m}_1 \dt^{M-m+1} \norm{J}^{M-m+1}_{\infty}\ 
\sum_{
\substack{\norm{\vect{\alpha}}_{\infty} \le M-m\\ \norm{\vect{\beta}}_{\infty} \le M-m\\ \sum_{j=0}^m \alpha_j + \beta_j \ge M-m+1}} 
\frac{1}{\prod_{j=0}^{m}\alpha_j! \prod_{j=0}^{m}\beta_j!}.
\end{align*}
\end{proof}

By combing errors from various level $m\le M-1$ in \propref{prop::quad_error_bound}, we immediately have the followings:
\begin{proposition}
\label{prop::error_quadrature}
Suppose that $\dt\norm{J}_\infty \le 1$. Suppose that for any $M\in \Natural$, the choice of quadrature scheme at each level $m$ with $1\le m\le M-1$ (whose weights are denoted as $w_{m,q}$ and whose samples are denoted as $\vect{s}^{m,q}$ with sample number index $q\le Q_m$) satisfies the consistency criterion \eqref{eqn::criterion}. Then the total error 
\begin{align}
\label{eqn::error_quadrature}
\begin{aligned}
& \norm{\sum_{m=1}^{M} \dt^{m} \Big(\int_{\cube_m} \intterm_{m}^{M}(s_{m}, \cdots, s_1) (\rho_0) \dd \vect{s}_{1:m}  - \sum_{q=1}^{Q_m} w_{m,q} \intterm_{m}^{M}(\vect{\s}^{m,q}) (\rho_0)\Big)}_{1}\\
& \le \sum_{m=1}^{M-1} \frac{2 (M-m)!}{M!} C(M,m) \norm{\lbop_L}_1^{m} \norm{J}_{\infty}^{M-m+1}  \dt^{M+1}.
\end{aligned}
\end{align}

\end{proposition}
As a remark, the upper bound of $C(M, m)$ comes from \eqref{eqn::CMm}, in particular, we remark that the upper bound on the right hand side of \eqref{eqn::CMm} can be explicitly and easily computed.

\subsection{Error from normalization}
Finally, we deal with the normalization error.

\begin{lemma}[Normalization error]
\label{lem::normalization_error}
Suppose $\UA$ is an (unnormalized) positivity-preserving linear scheme, such that for all $0< \dt\le \delta_0$ and any density matrix $\rho_0$, we have
\begin{align*}
    \ttnormbig{e^{\lbop \dt}(\rho_0) - \UA(\rho_0)} \le c_M \dt^{M+1},
\end{align*}
where $\delta_0$ and $c_M$ are some positive constants.
Then for any $0 < \dt \le \delta_0$,
we have 
\begin{align*}
    \ttnormbig{e^{\lbop \dt}(\rho_0) - \NA(\rho_0)} \le 2 c_M \dt^{M+1}.
\end{align*}
Moreover, for any fixed time $T > 0$ and any density matrix $\rho_0$, if $N \ge T /\delta_0$, then
\begin{align*}
\begin{aligned}
    \ttnormbig{e^{\lbop T}(\rho_0) - (\alg_{\frac{T}{N}})^{N}(\rho_0)} \le 2 c_M T^{M+1} N^{-M}.
\end{aligned}
\end{align*}
\end{lemma}

This lemma shows that for sufficiently small $\dt$, the normalization process only results in a constant prefactor, which is uniformly bounded by $2$.

\begin{proof}
Recall that $\NA(\rho) := \frac{\UA(\rho)}{\tr\bigl(\UA(\rho)\bigr)}$ for any density matrix $\rho$. Then for any density matrix $\rho_0$, 
\begin{align*}
    \ttnormbig{e^{\lbop\dt}(\rho_0) - \NA(\rho_0)} &\le \ttnormbig{e^{\lbop\dt}(\rho_0) - \UA(\rho_0)} + \ttnormbig{\UA(\rho_0) - \NA(\rho_0)}\\
    &\le  c_M \dt^{M+1} + \ttnormbig{\UA(\rho_0)} \cdot  \abs\bigg{\frac{1-\tr\bigl(\UA(\rho_0)\bigr)}{\tr\bigl(\UA(\rho_0)\bigr)}}.
\end{align*}
By assumption, we could straightforwardly observe that
\begin{align*}
\tr\bigl(\UA(\rho_0)\bigr) = \ttnormbig{\UA(\rho_0)} \in [1- c_M \dt^{M+1}, 1 + c_M \dt^{M+1}].
\end{align*}
Therefore,
\begin{align*}
    \ttnormbig{e^{\lbop\dt}(\rho_0) - \NA(\rho_0)} &\le c_M \dt^{M+1} +  c_M \dt^{M+1} = 2 c_M \dt^{M+1}.
\end{align*}
The finite-time error follows by a standard telescoping sum estimate:
\begin{align*}
    &\ttnormbig{e^{\lbop T}(\rho_0) - (\alg_{\frac{T}{N}})^{N} (\rho_0)} = \ttnormBig{\sum_{j=0}^{N-1} (e^{\lbop \dt})^{j} \circ (e^{\lbop \dt} - \NA)\circ (\NA)^{(N-j-1)} (\rho_0)} \\
&\le \sum_{j=0}^{N-1} \ttnormbig{(e^{\lbop \dt})^{j} \circ (e^{\lbop \dt} - \NA)\circ (\NA)^{(N-j-1)} (\rho_0)} \\
&= \sum_{j=0}^{N-1} \ttnormbig{(e^{\lbop \dt} - \NA)\circ (\NA)^{(N-j-1)} (\rho_0)}\\
&\le \sum_{j=0}^{N-1} 2 c_M \dt^{M+1} \le  2 c_M \dt^{M+1} N =  \frac{2 c_M T^{M+1}}{N^M},
\end{align*}
where we used the fact that Lindblad evolution is completely positive and trace-preserving in the second last line.
\end{proof}

\subsection{Proof of \thmref{thm::error_anal}}

A combination of \propref{prop::step_one_two_bound} and \propref{prop::error_quadrature} give us the total error of any $M^{\text{th}}$ order un-normalized scheme
\begin{align*}
&\norm{\rho_{\dt} - \alg_{\dt}^{(\text{un}, M)}(\rho_0)} \\
\le &\Big(\frac{23}{(M+1)!} \big(\norm{J}_{\infty} + \norm{\lbop_L}_{1}\big)^{M+1} +\sum_{m=1}^{M-1} \frac{2 (M-m)!}{M!} C(M,m) \norm{\lbop_L}_1^{m} \norm{J}_{\infty}^{M-m+1}\Big)  \dt^{M+1}.
\end{align*}
By \lemref{lem::normalization_error}, the conclusion for normalized scheme easily follows.

\end{document}